\documentclass{amsart}
\usepackage[backref]{hyperref}
\hypersetup{hidelinks}
\newtheorem{theorem}{Theorem}[section]
\newtheorem{lemma}{Lemma}

\newtheorem{corollary}{Corollary}

\theoremstyle{definition}

\theoremstyle{remark}
\newtheorem{remark}{Remark}

\numberwithin{equation}{section}

\usepackage{cleveref}
\usepackage{graphicx}
\usepackage{epstopdf}
\usepackage{cite}
\usepackage{bm}
\usepackage{multirow}
\usepackage{subfigure}
\crefname{theorem}{Theorem}{Theorems}
\crefname{equation}{\!\!}{\!\!}
\crefname{lemma}{Lemma}{Lemmas}
\crefname{figure}{Fig.}{Figs.}
\Crefname{equation}{Eq.}{Eqs.}
\crefname{example}{Example}{Examples}
\crefname{remark}{Remark}{Remarks}
\crefname{table}{Table}{Tables}
\crefname{assumption}{Assumption}{Assumptions}

\begin{document}

\title{A low order divergence-free H(div)-conforming finite element method for Stokes flows}

\author{Xu Li}
\address{School of Mathematics, Shandong University, Jinan 250100, China}
\email{xulisdu@126.com}
\thanks{The second author is the corresponding author.}

\author{Hongxing Rui}
\address{School of Mathematics, Shandong University, Jinan 250100, China}
\email{hxrui@sdu.edu.cn}

\begin{abstract}
In this paper, we propose a ${\bm P_{1}^{c}}\oplus {RT0}-P0$ discretization of the Stokes equations {on} general simplicial meshes in two/three dimensions (2D/3D), which yields an exactly divergence-free and pressure-independent velocity approximation with optimal order. {Our method has the following features}. Firstly, the global number of the degrees of freedom of our method is the same as the low order {Bernardi and Raugel ($B$-$R$) finite element method} (Bernardi and Raugel, 1985), while the number of {the non-zero entries} of the former is about half of the latter in the velocity-velocity region of the coefficient matrix. Secondly, the ${\bm P_{1}^{c}}$ component of the velocity, the $RT0$ component of the velocity and the pressure seem to solve a popular ${\bm P_{1}^{c}}-{RT0}-P0$ discretization of a poroelastic-type system formally. Finally, our method can be easily transformed into a pressure-robust and stabilized ${\bm P_{1}^{c}}-P0$ discretization for {the Stokes problem} via the static condensation of the $RT0$ component, {which has a much smaller number of global degrees of freedom}. Numerical experiments illustrating the robustness of our method are also provided.
\end{abstract}

\maketitle

\section{Introduction}
\label{sec:intro}
The importance of strong mass conservation has been found in a variety of real-world applications (see, e.g., \cite{linke_collision_2009,Linke2012,linke_velocity_2016}). For {the incompressible Stokes problem}, pointwise mass conservation requires an exactly divergence-free velocity approximation. Another highly relevant concept is {pressure-robustness} (see, e.g., \cite{john_divergence_2017,linke_pressure-robustness_2016,Gauger2019}), which describes a type of robustness that the velocity error is independent of the pressure. In this paper, we design a low order divergence-free and pressure-robust $H(\operatorname{div})$-conforming mixed {method} for {the Stokes problem} on general shape-regular simplicial meshes.

Due to the fascinating properties of divergence-free {mixed finite element methods} \cite{Evans20132,john_divergence_2017}, constructing efficient divergence-free methods has been a hot research {topic} during recent years and great {progress} has been made. However, this is still not an easy thing especially in three dimensions. For conforming divergence-free methods, we refer the readers to \cite{Scott1985,Arnold1992,Zhang2005,Neilan2014,Neilan2018,christiansen_generalized_2018,Evans2013,Evans20132}, to name just a few. Among these methods, the lowest order cases on simplicial meshes might be the ones proposed in \cite{christiansen_generalized_2018} and \cite{Neilan2018}, where {they have the same number of global degrees of freedom (DOFs) as the classical Bernardi and Raugel finite element method} \cite{bernardi_analysis_1985}. {Like the Bernardi and Raugel element, they both use $d+1$ face bubbles to enrich the linear polynomial space on each simplex in $d$ dimensions. The difference lies in the construction of the face bubbles: the bubbles in} \cite{christiansen_generalized_2018} {are piecewise linear on the Powell-Sabin split (e.g., in 3D, a tetrahedron is split into $12$ small tetrahedrons) of a simplex, and the bubbles in} \cite{Neilan2018} {are typically piecewise polynomials of degree $d$ on the barycentric refinement of a simplex.} {Some applications of the former could be found} in \cite{Burman2020}. Another popular class of mixed methods is based on $H(\operatorname{div})$-conforming elements, which {are} a class of nonconforming elements for {the} Stokes equations. In general, there are two ways to stabilize the {nonconformity} (tangential discontinuity) of the elements. One is to modify the velocity-velocity bilinear form into a discontinuous Galerkin (DG) framework (see, e.g., \cite{Cockburn2007,wang_new_2007,konno2011,Kanschat2014}), the other is to modify the discrete velocity space to
impose tangential continuity in some weak sense such as \cite{johnny_family_2012,Mardal2002,2008UNIFORMLY,Xue2006}. Among these methods the lowest order case on simplicial meshes might be {the DG method based on $BDM1-P0$ pair} \cite{Cockburn2007,wang_new_2007} ($BDM1$: the lowest order Brezzi-Douglas-Marini element \cite{brezzi_two_1985}, $P0$: the piecewise constant space with zero {mean}). Recently, {there is also a class of pressure-robust discretizations developed on classical non-divergence-free mixed methods via divergence-free reconstruction of test functions} (cf. \cite{Linke2014on,linke_robust_2016,linke_pressure-robustness_2016,Linke2017}), which can be seen as a class of Petrov-Galerkin methods.

{One of the advantages of DG $H(\operatorname{div})$-conforming mixed methods lies in the simple construction of their velocity spaces. For instance, the lowest order velocity is piecewise linear in arbitrary dimensions.} However, in a DG framework we should introduce some integral terms over element interfaces which may increase the cost and enlarge the stencil (sparsity). {In addition}, the stabilization parameters (the parameters of the jump-penalty term) need to be chosen carefully to guarantee the stability in some cases. The features of conforming and the second type of $H(\operatorname{div})$-conforming divergence-free mixed methods are exactly opposite to {the DG methods mentioned above}. They do not need to modify the bilinear form and {are} always stable. The constructions of their spaces are, however, a little complicated and high order polynomials are usually included. In a word, constructing divergence-free mixed methods is non-trivial.

The aim of this paper is to investigate a way on how to combine some advantages of the conforming and (DG-) $H(\operatorname{div})$-conforming divergence-free mixed methods. In our method, the velocity space is obtained by enriching the continuous vector-valued piecewise linear polynomial space (${\bm P_{1}^{c}}$) with the lowest order Raviart-Thomas-N\'{e}d\'{e}lec space ($RT0$) \cite{nedelec1980,galligani_mixed_1977}, and the pressure is approximated by piecewise constants with zero {mean}. In a sense, our method can be seen as replacing the Bernardi and Raugel bubbles \cite{bernardi_analysis_1985} with the lowest order Raviart-Thomas-N\'{e}d\'{e}lec elements (see \cref{fig:1}). This {replacement} does bring some advantages. For example, the {discrete velocity} becomes pointwise divergence-free. However, this {turns into} a type of nonconforming method and some stabilization strategies are needed. To resolve this {issue},
\begin{figure}[htbp]
\centering
\includegraphics[scale=0.25]{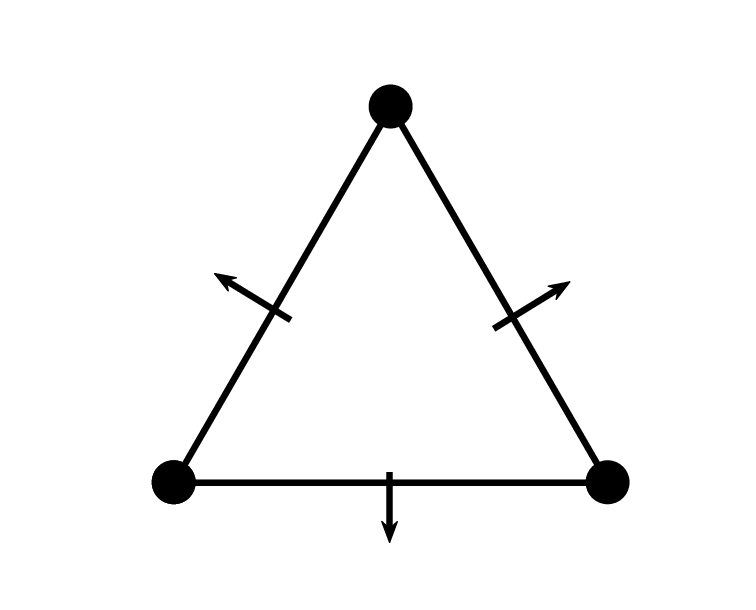}
\includegraphics[scale=0.25]{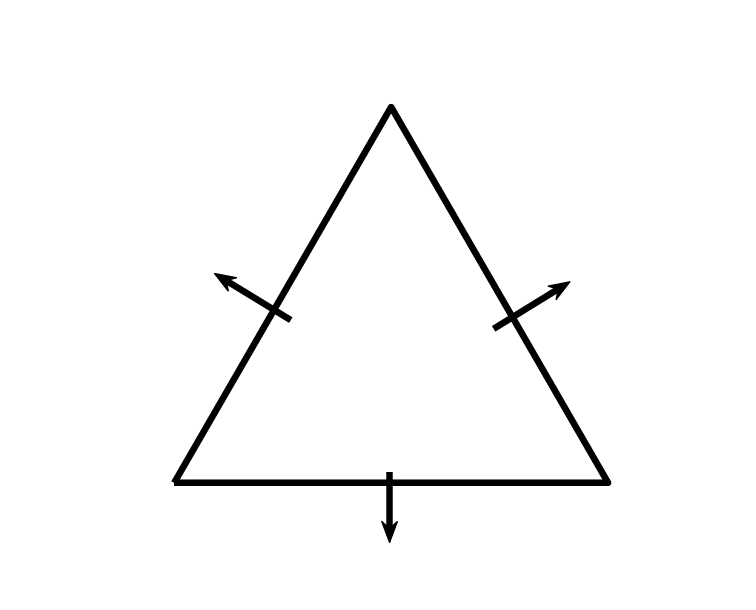}
\includegraphics[scale=0.25]{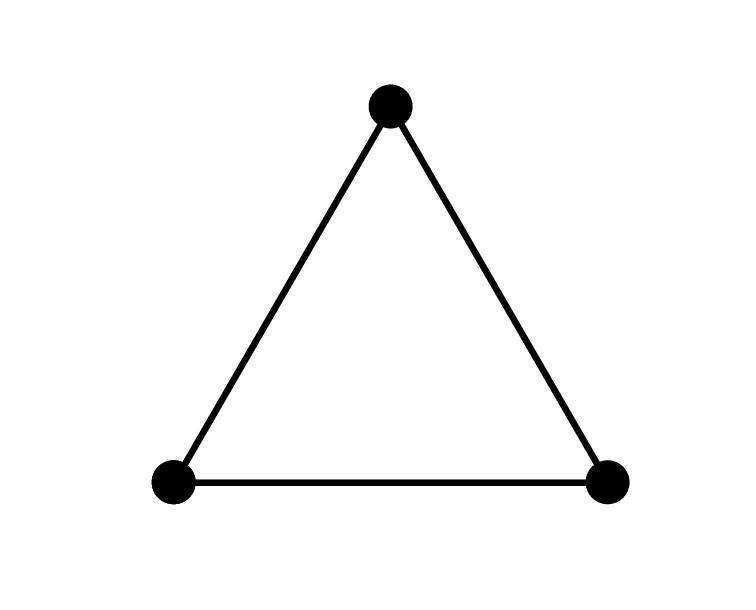}
\caption{{\scriptsize DOFs of $B$-$R$ (left), $RT0$ (middle) and ${\bm P_{1}^{c}}$ (right) element in 2D. Solid circles indicate function evaluation and arrows indicate normal component evaluation.}}
\label{fig:1}
\end{figure}
we propose a conforming-like formulation which does not involve any integrals over element interfaces. The resulting method has several advantages. Firstly, the construction of the space is very simple in any dimensions. Secondly, we can use the lowest order quadrature rule to compute each {entry} of the coefficient matrix exactly in some cases. Thirdly, our method instead has much better sparsity than the Bernardi and Raugel method. Next, the discrete system is stable as long as the parameters {introduced} are positive. Finally, we can further obtain a stabilized pressure-robust ${\bm P_{1}^{c}}-P0$ discretization for {the} Stokes equations {via static condensation}. Compared to a direct ${\bm P_{1}^{c}}-P0$ discretization (unstable), this only changes the pressure region of the right-hand side and the pressure-pressure region of the coefficient matrix, which is different from the similar strategy for {the} Bernardi and Raugel element in \cite{rodrigo_new_2018}. The price to pay for these advantages is a consistency error {bounded in an optimal order}.

{The motivation for studying low order divergence-free elements is threefold: (1) the low order element in this paper is easy to implement; (2) in some practical applications, the available mesh might be fixed due to the complex geometry} \cite{john_finite_2018}, {in which case lower order methods allow a smaller global system; (3) the divergence-free nature usually makes low order elements also give a satisfactory result for some non-trivial Stokes equations such as those with large pressure gradient, small viscosity or long time simulation (unsteady case), cf.} \cite{john_divergence_2017,Linke2018}.

{We also want to mention that hybridization is a very efficient way to reduce the computational cost of DG methods}, cf. \cite{Cockburn2009}. In \cite{Lehrenfeld2010} and \cite{lehrenfeld_high_2016}, a class of $H(\operatorname{div})$-conforming hybrid DG (HDG) methods with discontinuous pressure elements was presented. A cheaper version of them with relaxed $H(\operatorname{div})$-conformity for velocity and pressure-robust reconstructions could be found in \cite{lederer_hybrid_2018} and \cite{lederer_hybrid_2019}, where the lowest order pair (after static condensation) has one DOF per facet and dimension for velocity and one DOF per element for pressure, the same as $CR1-P0$ pair ($CR1$: the first order nonconforming Crouzeix-Raviart element, cf. \cite{CR1973}). A total hybridization (simultaneously for velocity and pressure) formulation for the Stokes equations was analyzed in \cite{Rhebergen2017}, where the velocity is automatically $H(\operatorname{div})$-conforming if the order of the pressure space is one less than the order of the velocity space. To further lower the cost, an embedded DG (EDG) method with pressure-robust reconstruction and an embedded-hybridized DG (EDG-HDG, i.e., EDG for velocity and HDG for pressure) method were developed in \cite{Lederer2020} and \cite{Rhebergen2020}, respectively. EDG methods use continuous facet unknowns which have the same number of DOFs as the corresponding conforming methods, while retaining the versatility of DG methods. In a sense, our method shares a similar topic with EDG methods, i.e., combining some advantages of conforming methods and DG methods. The number of the DOFs (in the reduced version) of the lowest order pair in \cite{Lederer2020} is equal to the one of ${\bm P_{1}^{c}}-P_{1}^{c}$ pair, where $P_{1}^{c}$ is the scalar version of ${\bm P_{1}^{c}}$. This is also equal to that of the lowest order MINI elements with static condensation of the cell bubbles, which can also be pressure-robust with reconstructions, cf. \cite{Linke2017}. The two low order pairs have less DOFs than ${\bm P_{1}^{c}}-P0$ due to a smaller pressure space and can be easily extended to higher order cases. Compared to them, the advantage of our method lies in its automatically divergence-free nature.

The rest of the paper is organized as follows. \Cref{sec:2} is devoted to describing the model and its discretizations. In \Cref{sec:3}, we analyze the LBB stability, consistency and error estimates. Some numerical studies are shown in \Cref{sec:5}. Finally we do some conclusions in \Cref{sec:6}.

Throughout the paper we use $C,$ with or without subscript, to denote a generic positive constant. The standard inner product and norm (seminorm) of the Sobolov space $[H^{m}(X)]^{n}$ or $[H^{m}(X)]^{n\times n}$ ($n\in\mathbb{N}^{+}$) are denoted by $(\cdot, \cdot)_{m,X}$ and $\|\cdot\|_{m,X}$ ($|\cdot|_{m,X}$), respectively. When $m=0$ ($X=\Omega$), with the convention that the index $m$ ($X$, respectively) is omitted. For any sub-dimensional face $e$, we have $\langle\cdot,\cdot\rangle_{e}$ and $||\cdot||_{e}$ similarly.
\section{Model, notation and method}
\label{sec:2}
In this section we consider the incompressible Stokes equations in a bounded domain $\Omega\subset \mathbb{R}^{d}~(d=2,3)$
\begin{subequations}\label{Stokes}
\begin{align}
-\nu \Delta \boldsymbol{u}+\nabla p&=\boldsymbol{f} \quad \text { in } \Omega,\label{Stokes1}\\
\nabla \cdot \boldsymbol{u}&=0 \,\quad \text { in } \Omega,\label{Stokes2}\\
\boldsymbol{u}&={\boldsymbol{0}} \quad \text { on } \partial \Omega,
\end{align}
\end{subequations}
where the unknowns $(\boldsymbol{u},p)$ represent the velocity and pressure, respectively; $\nu$ denotes the constant fluid viscosity; $\boldsymbol{f}$ is the external body force; \Cref{Stokes1} is usually called the momentum equation and \Cref{Stokes2} is usually called the continuity equation. For simplicity, we assume that $\Omega$ has a Lipschitz-continuous {polygonal/polyhedral boundary} $\Gamma$.

Let $V=[H_{0}^{1}(\Omega)]^{d}$ and $W=L_{0}^{2}(\Omega)$. A commonly used variational form of \Cref{Stokes} {is:}
\begin{subequations}\label{formula1}
\begin{align}
{\rm find}~(\boldsymbol{u},p)\in V\times W &~ {\rm such ~that}\nonumber\\
\nu a(\boldsymbol{u}, \boldsymbol{v})-b(\boldsymbol{v}, p) &=(\boldsymbol{f}, \boldsymbol{v}) \quad\forall~\boldsymbol{v}\in V,\\
b(\boldsymbol{u}, q) &=0\qquad\quad\forall~q\in W,
\end{align}
\end{subequations}
where
$$
a(\boldsymbol{u},\boldsymbol{v})=(\nabla\boldsymbol{u},\nabla\boldsymbol{v}),\quad
b(\boldsymbol{v}, q)=(\nabla\cdot\boldsymbol{v},q).
$$
And the spaces $(V,W)$ satisfy the following well-known inf-sup condition (See \cite[\S 4.2]{boffi_mixed_2013})
\begin{equation}\label{continuousInfSup}
\sup _{\boldsymbol{v} \in V\setminus\{\boldsymbol{0}\}} \frac{b(\boldsymbol{v}, q)}{||\boldsymbol{v}||_{1}} \geq \beta\|q\| \quad \forall~ q \in W,
\end{equation}
where $\beta$ is a positive constant.
\subsection{Discretizations}
\label{subsec:21}
Let $\mathcal{T}_{h}$ be shape-regular triangular or tetrahedral partitions of $\Omega$ \cite{Ciarlet2002The}. Let $h_{T}$ and $h_{e}$ denote the diameters of elements $T$ and faces $e$, respectively, and $h=\max_{T\in\mathcal{T}_{h}}h_{T}$. We denote the set of interior faces of $\mathcal{T}_{h}$ by $\mathcal{E}^{0}$, the set of boundary faces by $\mathcal{E}^{\partial}$ and $\mathcal{E}=\mathcal{E}^{0}\cup\mathcal{E}^{\partial}$.
We define
\begin{displaymath}
H(\operatorname{div} ; \Omega)=\left\{\boldsymbol{v} \in [L^{2}(\Omega)]^{d} : \nabla\cdot \boldsymbol{v} \in L^{2}(\Omega)\right\},\quad
H_{0}(\operatorname{div} ; \Omega)=\left\{\boldsymbol{v} \in H(\operatorname{div} ; \Omega) :\left.\boldsymbol{v} \cdot \boldsymbol{n}\right|_{\Gamma}=0\right\}.\end{displaymath}

{Let $P_{k}(T)$ ($k\geq0$) denote the space of polynomials of degree no more than $k$ on elements $T$,} and the finite element spaces are set as
\begin{displaymath}
V_{h}^{1}=\{\boldsymbol{v}\in V : \boldsymbol{v}|_{T}\in [P_{1}(T)]^{d}\quad \forall\ T  \in \mathcal{T}_{h}\},
\end{displaymath}
\begin{equation}\label{RT0}
RT0=\{\boldsymbol{v}\in H_{0}(\operatorname{div};\Omega) : \boldsymbol{v}|_{T}\in [P_{0}(T)]^{d}\oplus\boldsymbol{x}P_{0}(T)\quad \forall\ T  \in \mathcal{T}_{h}\},
\end{equation}
\begin{displaymath}
W_{h}=\left\{q \in L_{0}^{2}(\Omega):\left.q\right|_{T} \in P_{0}(T) \quad \forall~ T \in \mathcal{T}_{h}\right\}.
\end{displaymath}

We note that the functions in $V_{h}^{1}$ are totally continuous, while the functions in $RT0$ may have discontinuous tangential components across the interelement boundaries. The following lemma depicts a simple but important fact in our method.
\begin{lemma}\label{basiclemma}
$V_{h}^{1}\cap RT0=\{\bf 0\}$.
\end{lemma}
\begin{proof}
Consider a simplest case in two dimensions (see \cref{fig:basic}).
\begin{figure}
\centering
\includegraphics[width=3.0cm,height=3cm]{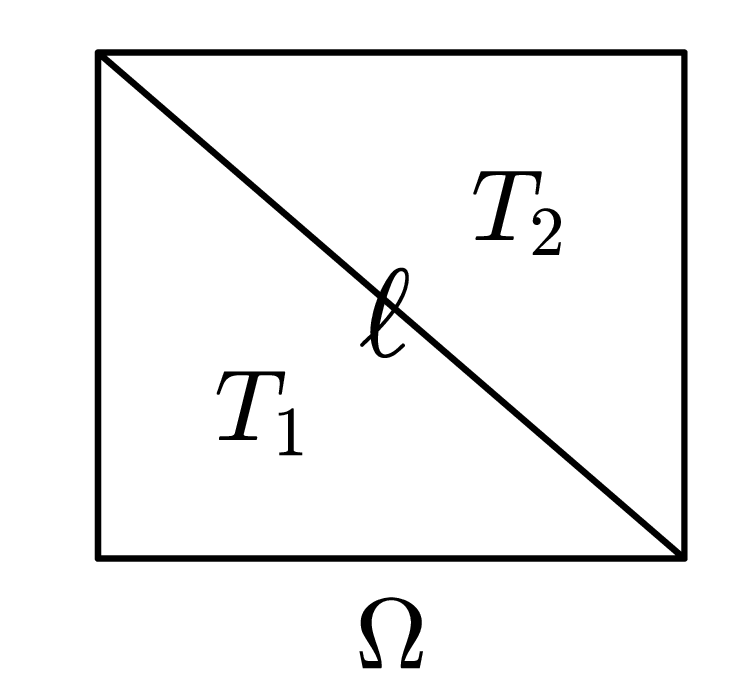}
\caption{Decompose $\Omega$ into two elements $T_{1}$ and $T_{2}$, which share an edge $\ell$.}
\label{fig:basic}
\end{figure}
Two axes are respectively labeled as $x$ and $y$. From \cref{RT0}, a function $\boldsymbol{f}^{rt}\in RT0$ can be written as follows,
\begin{equation}\nonumber
\boldsymbol{f}^{rt}|_{T_{1}}=
\begin{pmatrix}
a_{1}+c_{1}x, b_{1}+c_{1}y
\end{pmatrix}^{\top},
\quad
\boldsymbol{f}^{rt}|_{T_{2}}=
\begin{pmatrix}
a_{2}+c_{2}x, b_{2}+c_{2}y
\end{pmatrix}^{\top},\quad a_{i},b_{i},c_{i}\in \mathbb{R},\quad i=1,2.
\end{equation}
It suffices to prove that $\boldsymbol{f}^{rt}\in V_{h}^{1}\Rightarrow\boldsymbol{f}^{rt}=\boldsymbol{0}$.
If $\boldsymbol{f}^{rt}\in V_{h}^{1}$, $\boldsymbol{f}^{rt}$ is totaly continuous across the element interface $\ell$. In this case, we have
\begin{subequations}
\begin{align}
\label{a1plusc1x}a_{1}+c_{1}x=a_{2}+c_{2}x \quad \text{on}~ \ell,\\
\label{b1plusc1y}b_{1}+c_{1}y=b_{2}+c_{2}y \quad \text{on}~ \ell.
\end{align}
\end{subequations}
We note that there is at least one of the variables ($x$ or $y$) varying over $\ell$. For brevity, we assume that $x$ varies over $\ell$. From \cref{a1plusc1x} we can obtain
\begin{displaymath}
a_{1}=a_{2},\quad c_{1}=c_{2}.
\end{displaymath}
Substituting $c_{1}=c_{2}$ into \cref{b1plusc1y} gives
{$b_{1}=b_{2}.$}
Then we have
{\begin{equation}\nonumber
\boldsymbol{f}^{rt}|_{\Omega}=
\begin{pmatrix}
a_{1}+c_{1}x, b_{1}+c_{1}y
\end{pmatrix}^{\top}.
\end{equation}}
If we further demand that $\boldsymbol{f}^{rt}|_{\Gamma}=\boldsymbol{0}$, through the similar analysis, one can obtain
{$a_{1}=b_{1}=c_{1}=0.$} Thus we complete the proof of this simplest case. The proofs of general decompositions and three-dimensional case are a natural extension of this.
\end{proof}

{Denote by $V_{h}=V_{h}^{1}\oplus RT0$ the discrete velocity space.} Let $\mathcal{RT}:=\{\boldsymbol{\Phi}_{e}, ~e\in \mathcal{E}^{0}\}$ denote the set of the commonly used bases in $RT0$, where $\boldsymbol{\Phi}_{e}$ is only supported on two elements $T_{1}$ and $T_{2}$ which share the face $e$ and $\boldsymbol{\Phi}_{e}\cdot\boldsymbol{n}_{e'}$ vanishes on any face $e'\neq e$. Here $\boldsymbol{n}_{e'}$ denotes the unit normal vector to $e'$.

For any $\boldsymbol{v}_{h}\in V_{h}$, we have the unique decompositions
{\begin{displaymath}
\boldsymbol{v}_{h}=\boldsymbol{v}_{h}^{1}+\boldsymbol{v}_{h}^{R},\quad \boldsymbol{v}_{h}^{R}=\sum_{e\in\mathcal{E}^{0}}v_{e}\boldsymbol{\Phi}_{e} \quad \text { with }\boldsymbol{v}_{h}^{1}\in V_{h}^{1},
\ \boldsymbol{v}_{h}^{R}\in RT0.
\end{displaymath}

Our discretization is presented as:
\begin{subequations}\label{formula2}
\begin{align}
{\rm Find}~(\boldsymbol{u}_{h},p_{h})\in V_{h}\times W_{h}~ &{\rm such ~that}\nonumber\\
\nu a_{h}(\boldsymbol{u}_{h}, \boldsymbol{v}_{h})-b(\boldsymbol{v}_{h}, p_{h}) &=(\boldsymbol{f}, \boldsymbol{v}_{h})
\quad\forall~\boldsymbol{v}_{h}\in V_{h},\label{formula21} \\
b(\boldsymbol{u}_{h}, q_{h}) &=0\qquad\ \, \quad\forall~q_{h}\in W_{h},\label{formula22}
\end{align}
\end{subequations}
where
\begin{equation}\label{formula2aux}
a_{h}(\boldsymbol{u}_{h}, \boldsymbol{v}_{h})=a(\boldsymbol{u}_{h}^{1},\boldsymbol{v}_{h}^{1})+a^{R}(\boldsymbol{u}_{h}^{R}, \boldsymbol{v}_{h}^{R}).
\end{equation}
Here $a^{R}: RT0\times RT0\rightarrow R$ has several choices as
\begin{equation}\label{mathcalJ0}
a^{R}(\boldsymbol{u}_{h}^{R}, \boldsymbol{v}_{h}^{R})=a^{0}(\boldsymbol{u}_{h}^{R}, \boldsymbol{v}_{h}^{R})
:=\sum_{T\in\mathcal{T}_{h}}\alpha_{T} h_{T}^{-2}(\boldsymbol{u}_{h}^{R}, \boldsymbol{v}_{h}^{R})_{T},
\end{equation}
or
\begin{equation}\label{mathcalJD}
a^{R}(\boldsymbol{u}_{h}^{R}, \boldsymbol{v}_{h}^{R})=a^{D}(\boldsymbol{u}_{h}^{R}, \boldsymbol{v}_{h}^{R})
:=\sum_{T\in\mathcal{T}_{h}}\sum_{e\in\partial T\cap\mathcal{E}^{0}}\alpha_{T}h_{T}^{-2}u_{e}v_{e}(\boldsymbol{\Phi}_{e},\boldsymbol{\Phi}_{e})_{T},
\end{equation}
or
\begin{equation}\label{mathcalJd}
a^{R}(\boldsymbol{u}_{h}^{R}, \boldsymbol{v}_{h}^{R})=a^{\operatorname{div}}(\boldsymbol{u}_{h}^{R}, \boldsymbol{v}_{h}^{R})
:=\sum_{T\in\mathcal{T}_{h}}\sum_{e\in\partial T\cap\mathcal{E}^{0}}\alpha_{T}u_{e}v_{e}(\nabla\cdot\boldsymbol{\Phi}_{e},\nabla\cdot\boldsymbol{\Phi}_{e})_{T},
\end{equation}
where $\alpha_{T}$, $T\in\mathcal{T}_{h}$ are positive parameters.

Note that
\begin{displaymath}
a^{0}(\boldsymbol{u}_{h}^{R}, \boldsymbol{v}_{h}^{R})=\sum_{T\in\mathcal{T}_{h}}\sum_{e\in\mathcal{E}^{0}\cap \partial T}\sum_{e'\in\mathcal{E}^{0}\cap\partial T}\alpha_{T}
h_{T}^{-2}u_{e}v_{e'}(\boldsymbol{\Phi}_{e},\boldsymbol{\Phi}_{e'})_{T},
\end{displaymath}
and
\begin{displaymath}
a^{D}(\boldsymbol{u}_{h}^{R}, \boldsymbol{v}_{h}^{R})=\sum_{T\in\mathcal{T}_{h}}\sum_{e\in\mathcal{E}^{0}\cap \partial T}\sum_{e'\in\mathcal{E}^{0}\cap\partial T}{\delta}_{ee'}\alpha_{T}h_{T}^{-2}u_{e}v_{e'}(\boldsymbol{\Phi}_{e},\boldsymbol{\Phi}_{e'})_{T},
\end{displaymath}
where {${\delta}_{ee'}=1$ if $e=e'$, otherwise vanishes.} Thus, $a^{D}(\cdot,\cdot)$ can be regarded as a diagonalization version of $a^{0}(\cdot,\cdot)$. Clearly the matrices from $a^{D}(\cdot,\cdot)$ and $a^{\operatorname{div}}(\cdot,\cdot)$ are both diagonal. In the next section,
we will prove that the three choices of $a^{R}(\cdot,\cdot)$ are equivalent and formula \cref{formula2} is stable as long as $\alpha_{T}>0$ for all $T\in\mathcal{T}_{h}$.

\begin{remark}[The motivation of $a^{R}$]
The three forms of $a^{R}$ proposed here can also be regarded as some penalty-type stabilizations. From \cref{basiclemma} we can further obtain $RT0\cap V=\{\boldsymbol{0}\}$. In other words, there is no conforming element in $RT0$. Thus, we penalize the nonconformity in another way. That is, we penalize the whole $RT0$ component, and enforce $\boldsymbol{u}_{h}^{R}=\boldsymbol{0}$ when $h\rightarrow0$.
\end{remark}
The formulation \cref{formula2} can be {rewritten} to the following equivalent three-field form:
\begin{subequations}\label{formula2threefield}
\begin{align}
{\rm Find}~(\boldsymbol{u}_{h}^{1},\boldsymbol{u}_{h}^{R},p_{h})\in V_{h}^{1}\times RT0\times W_{h}~ &{\rm such ~that}\nonumber\\
\nu a(\boldsymbol{u}_{h}^{1}, \boldsymbol{v}_{h}^{1})-b(\boldsymbol{v}_{h}^{1}, p_{h}) &=(\boldsymbol{f}, \boldsymbol{v}_{h}^{1})
~\quad\forall~\boldsymbol{v}^{1}_{h}\in V_{h}^{1},\label{formula3field1} \\
\nu a^{R}(\boldsymbol{u}_{h}^{R}, \boldsymbol{v}_{h}^{R})-b(\boldsymbol{v}_{h}^{R}, p_{h}) &=(\boldsymbol{f}, \boldsymbol{v}_{h}^{R})
\quad\forall~\boldsymbol{v}^{R}_{h}\in RT0,\label{formula3field2} \\
b(\boldsymbol{u}_{h}^{1}, q_{h})+b(\boldsymbol{u}_{h}^{R}, q_{h}) &=0\qquad\ \, \quad\forall~q_{h}\in W_{h}.\label{formula3field3}
\end{align}
\end{subequations}
\begin{remark}
\Cref{formula3field1} is a ${\bm P_{1}^{c}}$ discretization of the momentum equation \cref{Stokes1}, while \cref{formula3field2} is very similar to a low order discretization of the Darcy equation with the permeability $\kappa=(\nu\alpha_{T}h_{T}^{-2})^{-1}$ on the element $T$, when $a^{0}$ is used. Thus, our method is very similar to a ${\bm P_{1}^{c}}-RT0-P0$ discretization of a three-field (displacement - (Darcy) velocity - pressure) poroelastic-type system \cite{phillips2007}.
\end{remark}

Since $\nabla\cdot RT0=W_{h}$ (see \cite[Proposition 2.3.3]{boffi_mixed_2013}) and $\nabla\cdot V_{h}^{1}\subset W_{h}$, we have $\nabla\cdot V_{h}=W_{h}$. Thus, from \cref{formula22} we know

\begin{theorem}[Mass Conservation]
The finite element solution $\boldsymbol{u}_{h}$ of \cref{formula2} features a full satisfaction of the continuity equation, which means
\begin{equation}\label{divergencefree}
\nabla\cdot\boldsymbol{u}_{h}\equiv0.
\end{equation}
In other words, our method conserves mass pointwise.
\end{theorem}
\begin{proof}
Taking $q_{h}=\nabla\cdot\boldsymbol{u}_{h}$ in \cref{formula22} gives
{$||\nabla\cdot\boldsymbol{u}_{h}||^{2}=0.$} Then \cref{divergencefree} follows.
\end{proof}

In the continuous case, there is a fundamental invariance {property:} Changing the external force by a gradient field changes only the pressure solution, and not the velocity; in symbols,
\begin{displaymath}
\boldsymbol{f} \rightarrow \boldsymbol{f}+\nabla \psi \quad \Longrightarrow \quad(\boldsymbol{u}, p) \rightarrow(\boldsymbol{u}, p+\psi),
\end{displaymath}
where $\psi\in H^{1}(\Omega)/\mathbb{R}$.
This phenomenon is highly relevant to {pressure-robustness} \cite{john_divergence_2017}. In our method, we have the following discrete invariance property.
\begin{theorem}
\label{invariance}
The following invariance property holds in the discrete scheme \cref{formula2}:
\begin{equation}\label{pressurerobust}
\boldsymbol{f} \rightarrow \boldsymbol{f}+\nabla \psi \quad \Longrightarrow \quad(\boldsymbol{u}_{h}, p_{h}) \rightarrow(\boldsymbol{u}_{h}, p_{h}+P_{h}\psi),
\end{equation}
where $P_{h}: W\rightarrow W_{h}$ is the $L^{2}$ projection (see \cref{L2orthogonal}).
\end{theorem}
\begin{proof}
Since $\nabla\cdot V_{h}= W_{h}$, we have
\begin{equation}\label{pressurerobust0}
(p_{h}+P_{h}\psi,\nabla\cdot\boldsymbol{v}_{h})=(p_{h}+\psi,\nabla\cdot\boldsymbol{v}_{h})=(p_{h},\nabla\cdot\boldsymbol{v}_{h})-(\nabla\psi,\boldsymbol{v}_{h}).
\end{equation}
Substituting \cref{pressurerobust0} into \cref{formula2} gives
\begin{subequations}\nonumber
\begin{align}
a_{h}(\boldsymbol{u}_{h},\boldsymbol{v}_{h})-b(\boldsymbol{v}_{h},p_{h}+P_{h}\psi)
&=(\boldsymbol{f}+\nabla\psi,\boldsymbol{v}_{h})\quad\forall~\boldsymbol{v}_{h}\in V_{h},\\
b(\boldsymbol{u}_{h}, q_{h}) &=0\qquad\quad\quad\quad\, \quad\forall~q_{h}\in W_{h},
\end{align}
\end{subequations}\nonumber
which implies that $(\boldsymbol{u}_{h}, p_{h}+P_{h}\psi)$ is a solution corresponding to $\boldsymbol{f}+\nabla \psi$. Together with the unique solvability of our methods (see next section), the invariance property \cref{pressurerobust} follows.
\end{proof}
\subsection{A pressure-robust and stabilized ${P_{1}^{c}}-P0$ scheme}
\label{sec:4}
A block form of the fully discrete problem \cref{formula2} or \cref{formula2threefield} can be written as follows,
\begin{equation}\label{blockform}
\mathcal{A}_{S}\begin{pmatrix}U_{L} \\U_{R} \\P\end{pmatrix}=\begin{pmatrix}F_{L} \\F_{R} \\ \boldsymbol{0}\end{pmatrix}, \quad \text { with }
\mathcal{A}_{S}=\begin{pmatrix}A_{LL} & 0 & -G_{L}^{T} \\0 & A_{RR} & -G_{R}^{T} \\G_{L} & G_{R} & 0\end{pmatrix},
\end{equation}
where ${U}_{L}, {U}_{R},$ and ${P}$ are the unknown vectors corresponding to the continuous linear component of the velocity, the $RT0$ component of the velocity, and the pressure, respectively. The blocks $A_{LL}$, $A_{RR}$, $G_{L}$, $G_{R}$, {$F_{L}$ and $F_{R}$} correspond to the bilinear forms $a(\boldsymbol{u}_{h}^{1},\boldsymbol{v}_{h}^{1})$, $a^{R}(\boldsymbol{u}_{h}^{R},\boldsymbol{v}_{h}^{R})$, $b(\boldsymbol{u}_{h}^{1},q_{h})$, $b(\boldsymbol{u}_{h}^{R},q_{h})$, $(\boldsymbol{f},\boldsymbol{v}_{h}^{1})$ and $(\boldsymbol{f},\boldsymbol{v}_{h}^{R})$, respectively.

For a classical Bernardi and Raugel discretization, the cross region of the linear component and the bubble component of the velocity in the coefficient matrix is non-zero. In other words, the matrices from $a(\boldsymbol{u}_{h}^{1},\boldsymbol{v}_{h}^{b})$ and $a(\boldsymbol{u}_{h}^{b},\boldsymbol{v}_{h}^{1})$ are not zero matrices, where $\boldsymbol{u}_{h}^{b}$ and $\boldsymbol{v}_{h}^{b}$ denote the bubble parts. Thus, the number of the non-zero entries in the velocity-velocity region from our method is about half of the Bernardi and Raugel finite element method.

When $a^{D}$ or $a^{\operatorname{div}}$ is used, the block $A_{RR}$ is diagonal. In this situation, the component $\boldsymbol{u}_{h}^{R}$ can be locally expressed in terms of the discrete pressure $p_{h}$. In other words, the inverse of $A_{RR}$ is easy to obtain and also diagonal (sparse). It is then easy to eliminate the unknowns corresponding to $RT0$ via static condensation, which leads to the following system:
{$$
\widehat{\mathcal{A}}_{S}^{D}\begin{pmatrix} U_{L} \\ P \end{pmatrix}
=\begin{pmatrix} F_{L}\\ -G_{R}A_{RR}^{-1}F_{R} \end{pmatrix} \quad
\text{ with }
\widehat{\mathcal{A}}_{S}^{D}=\begin{pmatrix}
A_{LL} & -G_{L}^{T}\\G_{L}& G_{R}A_{RR}^{-1} G_{R}^{T}\end{pmatrix}.
$$}
Since $\boldsymbol{u}_{h}$ is pressure-robust, $\boldsymbol{u}_{h}^{1}$ is also pressure-robust (see the error estimate section below). In the next section we will prove that although $\boldsymbol{u}_{h}^{1}$ is not divergence-free, it has optimal convergence properties in $L^{2}$ and $H^{1}$ norms. Thus, the above system can be regarded as a stabilized and pressure-robust ${\bm P_{1}^{c}}-P0$ discretization of the Stokes equations.
\begin{remark}
\label{remark3}
When $a^{R}=a^{\operatorname{div}}$, all the terms in the left-hand side of \cref{blockform} can be computed out exactly by the lowest order quadrature rule (barycentric quadrature rule). And the system can be reduced to a ${\bm P_{1}^{c}}-P0$ system. Thus, we prefer the bilinear form $a^{\operatorname{div}}$ in practice.
\end{remark}
\begin{remark}
To obtain a divergence-free solution from the above ${\bm P_{1}^{c}}-P0$ system, one can get $\boldsymbol{u}_{h}^{1}$ and $p_{h}$ first, then compute $\boldsymbol{u}_{h}^{R}$ by $p_{h}$. Since $A_{RR}$ is diagonal, the computation of $\boldsymbol{u}_{h}^{R}$ is very low-cost. Compared to a direct ${\bm P_{1}^{c}}-P0$ discretization (unstable) of the Stokes equations, our method only changes the pressure-pressure region (lower right) of the coefficient matrix and the pressure region of the right-hand side.
\end{remark}
\section{Error estimates}
\label{sec:3}
\subsection{Preliminaries}
\label{subsec:31}
{We define the standard Raviart-Thomas interpolation operator $\Pi_{h}^{R}: V \rightarrow RT0$ and the piecewise integral mean operator $P_{h}: W \rightarrow {W}_{h}$ by}
\begin{displaymath}
\left(\left(\boldsymbol{v}-{\Pi}_{h}^{R} \boldsymbol{v}\right)\cdot\boldsymbol{n}_{e}, 1\right)_{e}=0 \quad \forall~ e \in \mathcal{E},
\end{displaymath}
and
\begin{equation}\label{L2orthogonal}\left(r-P_{h} r, q\right)=0\quad \forall~ q \in {W}_{h},\end{equation}
respectively.
\begin{lemma}\label{interpolationlemma1}
The operators $\Pi_{h}^{R}$ and $P_{h}$ satisfy the following properties:
\begin{equation}\label{L2approximation0}
\|\boldsymbol{v}-\Pi_{h}^{R}\boldsymbol{v}\|_{T} \leq C_{R} h_{T} |\boldsymbol{v}|_{1,T}\quad\forall~ T\in\mathcal{T}_{h},
\end{equation}
\begin{equation}\label{L2projectionerror}
\left\|r-P_{h} r\right\|_{T} \leq C_{P} h_{T}^{m}|r|_{m,T}\quad\forall~ T\in\mathcal{T}_{h}, \quad m=0,1,
\end{equation}
and
\begin{equation}\label{divapproximation0}
\nabla\cdot\Pi_{h}^{R}\boldsymbol{v}=P_{h}\nabla\cdot\boldsymbol{v},
\end{equation}
for any $\boldsymbol{v}\in V$ and $r\in W\cap H^{m}(\Omega)$, where $C_{R}$ and $C_{P}$ depend on the shape regularity of $\mathcal{T}_{h}$.
\end{lemma}
\begin{proof}
See \cite[\S 2.5]{boffi_mixed_2013}.
\end{proof}

Next, let
{$ V(h)=V\oplus RT0.$} be a larger space containing $V$ and $V_{h}$.
Recall that $V=\left[H_{0}^{1}(\Omega)\right]^{d}$. For any $\boldsymbol{v}\in V(h)$, we have the unique decomposition
{$
\boldsymbol{v}=\boldsymbol{v}^{1}+\boldsymbol{v}^{R}
$
with $\boldsymbol{v}^{1}\in {V}, \boldsymbol{v}^{R}\in RT0.$}
An extension of the definition of $a_{h}(\cdot,\cdot)$, \cref{formula2aux}, to the whole $V(h)\times V(h)$ is needed prior to analysis. We define that, for any $\boldsymbol{w},\boldsymbol{v}\in V(h)$,
\begin{equation}\label{formula0}
a_{h}(\boldsymbol{w},\boldsymbol{v}):=a(\boldsymbol{w}^{1},\boldsymbol{v}^{1})+a^{R}(\boldsymbol{w}^{R},\boldsymbol{v}^{R}).
\end{equation}
Restricted to the finite element space $V_{h}\times V_{h}$, the definition \cref{formula0} coincides with the definition \cref{formula2aux}. And on $V\times V$ we have
{$a_{h}(\cdot,\cdot)|_{V\times V}=a(\cdot,\cdot).$} Thus the definition of $a_{h}(\cdot,\cdot)$ is logical and natural.

To investigate the properties of the discretization formula \cref{formula2}, we introduce a norm $|||\cdot|||$ for the set $V(h)$ as follows,
\begin{equation}\label{norm}
|||\boldsymbol{v}|||^{2}=a_{h}(\boldsymbol{v},\boldsymbol{v})\quad\forall ~\boldsymbol{v}\in V(h).
\end{equation}
In this definition, for any $\boldsymbol{v}^{1}\in {V}$, we have
{$|||\boldsymbol{v}^{1}|||=|\boldsymbol{v}^{1}|_{1}.$}

For the case that $a^{R}=a^{0}$, it is not difficult to verify that $\|\boldsymbol{v}\|\leq\|\boldsymbol{v}^{1}\|+\|\boldsymbol{v}^{R}\|\leq |\boldsymbol{v}^{1}|_{1}+\|\boldsymbol{v}^{R}\|\leq C |||\boldsymbol{v}|||$ for any $\boldsymbol{v}\in V(h)$. Thus, \cref{norm} really defines a norm for $a^{R}=a^{0}$. For the other two choices of $a^{R}$, that is, $a^{D}$ and $a^{\operatorname{div}}$, we have the following results.
\begin{lemma}\label{equivalence}
For any $\boldsymbol{v}^{R}\in RT0$, we have
\begin{equation}\label{inequality0}
C_{*}a^{0}(\boldsymbol{v}^{R},\boldsymbol{v}^{R})\leq a^{D}(\boldsymbol{v}^{R},\boldsymbol{v}^{R})\leq
C^{*}a^{0}(\boldsymbol{v}^{R},\boldsymbol{v}^{R}),
\end{equation}
and
\begin{equation}\label{inequalitydiv}
\gamma_{*}a^{\operatorname{div}}(\boldsymbol{v}^{R},\boldsymbol{v}^{R})\leq a^{D}(\boldsymbol{v}^{R},\boldsymbol{v}^{R})\leq
\gamma^{*}a^{\operatorname{div}}(\boldsymbol{v}^{R},\boldsymbol{v}^{R}),
\end{equation}
where the constant $C_{*}$ only depends on the parameters $\alpha_{T},T\in\mathcal{T}_{h}$, and $C^{*}$, $\gamma_{*}$, $\gamma^{*}$ only depend on the shape regularity of $\mathcal{T}_{h}$ and the parameters.
\end{lemma}
\begin{proof}
For simplicity, we will only prove the lemma in the case that $\alpha_{T}\equiv\alpha$ for all $T\in\mathcal{T}_{h}$, where $\alpha$ is an arbitrary but fixed positive constant. Let us prove the first inequality in \cref{inequality0}.
Applying the same technique in \cite[Lemma 4.2]{rodrigo_new_2018}, by the Cauchy-Schwarz inequality we can similarly obtain
\begin{equation}\label{diagonalieq}
\left(\boldsymbol{v}^{R},\boldsymbol{v}^{R}\right)_{T}\leq (d+1)\sum_{e\in\mathcal{E}^{0}\cap\partial T} v_{e}^{2}(\boldsymbol{\Phi}_{e},\boldsymbol{\Phi}_{e})_{T}\quad\forall~T\in\mathcal{T}_{h},
\end{equation}
{where $v_{e}$ is the coefficient of $\boldsymbol{\Phi}_{e}$ such that $\boldsymbol{v}^{R}=\sum_{e\in\mathcal{E}^{0}}v_{e}\boldsymbol{\Phi}_{e}$.} Summation of the above inequality over all the elements $T\in\mathcal{T}_{h}$ gives the first inequality in \cref{inequality0} with $C_{*}=1/(d+1)$.

Let us prove the second inequality of \cref{inequality0}.
For any basis $\boldsymbol{\Phi}_{e}\in \mathcal{RT}$, $e\subset\partial T$, since $\boldsymbol{\Phi}_{e}\cdot\boldsymbol{n}_{e'}$ vanishes on any face $e'\neq e$, we have
$\boldsymbol{\Phi}_{e}|_{e}\cdot\boldsymbol{n}_{e}=(|T|/|e|)\nabla\cdot\boldsymbol{\Phi}_{e}|_{T},$
where $|\cdot|$ denotes the measure and we use the divergence theorem here.
By a series of simple calculations we arrive at
\begin{equation}\label{tracelement}
\|\boldsymbol{\Phi}_{e}\cdot\boldsymbol{n}_{e}\|_{e}^{2}=(|T|/|e|)\|\nabla\cdot\boldsymbol{\Phi}_{e}\|_{T}^{2}.
\end{equation}
{From} \cite[Eqs. 2.1.75\&2.1.78]{boffi_mixed_2013} {and scaling arguments one can obtain}
\begin{equation}\label{L2andL2divergence}
||\boldsymbol{\Phi}_{e}||_{T}\leq C h_{T} ||\nabla\cdot\boldsymbol{\Phi}_{e}||_{T}.
\end{equation}
{Since we assume the partition is shape-regular, the combination of} \cref{tracelement} and \cref{L2andL2divergence} {gives that} {$||\boldsymbol{\Phi}_{e}||_{T}\leq C h_{e}^{1/2}||\boldsymbol{\Phi}_{e}\cdot\boldsymbol{n}_{e}||_{e}$.} Then for the elements $T_{1}$ and $T_{2}$ which share the face $e$, we have
{\begin{displaymath}
v_{e}^{2}||\boldsymbol{\Phi}_{e}||_{T_{i}}^{2}\leq C h_{e} v_{e}^{2}||\boldsymbol{\Phi}_{e}\cdot\boldsymbol{n}_{e}||_{e}^{2}=C h_{e} ||v_{e}\boldsymbol{\Phi}_{e}\cdot\boldsymbol{n}_{e}||_{e}^{2}
=C h_{e} ||\boldsymbol{v}^{R}\cdot\boldsymbol{n}_{e}||_{e}^{2}\leq C ||\boldsymbol{v}^{R}||_{T_{i}},
\end{displaymath}}
for $i=1,2$, where in the last inequality we use the trace inequality. Sum over all the elements and apply the fact that the mesh is shape-regular, then the second inequality in \cref{inequality0} can be proven.

Finally, the two inequalities in \cref{inequalitydiv} follow immediately from an inverse inequality and the summation of \cref{L2andL2divergence} over all the elements, respectively. Thus we complete the proof.
\end{proof}

\subsection{Approximation}
\label{subsec:32}
In this subsection, we will construct an interpolation operator $\Pi_{h}:V\cap [C^{0}(\bar{\Omega})]^{d}\rightarrow V_{h}$ which satisfies the following properties
\begin{equation}\label{commuting}
\nabla\cdot\Pi_{h}\boldsymbol{v}=P_{h}\nabla\cdot\boldsymbol{v}
\end{equation}
\begin{equation}\label{approximation1}
||\boldsymbol{v}-\Pi_{h}\boldsymbol{v}||_{T}+h_{T}|||\boldsymbol{v}-\Pi_{h}\boldsymbol{v}|||_{T}\leq {C_{I} h_{T}^{2}|\boldsymbol{v}|_{2,T},}
\end{equation}
for all $T\in\mathcal{T}_{h}, \boldsymbol{v}\in [H^{2}({\Omega})]^{d}$. Here $|||\cdot|||_{T}$ is defined as
\begin{equation}\nonumber
|||\boldsymbol{v}|||_{T}^{2}=a_{h}^{T}(\boldsymbol{v},\boldsymbol{v}),
\end{equation}
where $a_{h}^{T}(\cdot,\cdot)$ is defined in the natural sense such that $a_{h}(\cdot,\cdot)=\sum_{T\in\mathcal{T}_{h}}a_{h}^{T}(\cdot,\cdot)$.

To this end, denote by $\Pi_{h}^{1}:V\cap [C^{0}(\bar{\Omega})]^{d}\rightarrow V_{h}^{1}$ the usual nodal interpolation and we define $\Pi_{h}$ as
\begin{equation}\label{interpolation}
\Pi_{h}\boldsymbol{v}=\Pi_{h}^{1}\boldsymbol{v}+\Pi_{h}^{R}(\boldsymbol{v}-\Pi_{h}^{1}\boldsymbol{v})\quad\forall~\boldsymbol{v}\in V\cap [C^{0}(\bar{\Omega})]^{d}.
\end{equation}
For the properties of $\Pi_{h}^{R}$, see \cref{interpolationlemma1}.
 We have the following lemma:
\begin{lemma}\label{interpolationlemma}
The interpolation operator $\Pi_{h}: V\cap [C^{0}(\bar{\Omega})]^{d}\rightarrow V_{h}$ defined in \cref{interpolation} satisfies the properties \cref{commuting} and \cref{approximation1}.
\end{lemma}
\begin{proof}
Since
$
\nabla\cdot V_{h}^{1}\subset W_{h},
$
we have
\begin{equation}\label{equality}
P_{h}\nabla\cdot\boldsymbol{v}_{h}^{1}=\nabla\cdot\boldsymbol{v}_{h}^{1}\quad\forall~\boldsymbol{v}_{h}^{1}\in V_{h}^{1}.
\end{equation}
\Cref{divapproximation0,equality} imply
\begin{displaymath}
\nabla\cdot\Pi_{h}\boldsymbol{v}=\nabla\cdot\Pi_{h}^{1}\boldsymbol{v}+ \nabla\cdot\Pi_{h}^{R}(\boldsymbol{v}-\Pi_{h}^{1}\boldsymbol{v})=\nabla\cdot\Pi_{h}^{1}\boldsymbol{v}+P_{h} \nabla\cdot(\boldsymbol{v}-\Pi_{h}^{1}\boldsymbol{v})=P_{h}\nabla\cdot\boldsymbol{v},
\end{displaymath}
Thus \cref{commuting} has been proved.

Let us prove \cref{approximation1}. For simplicity, we will only prove it for $a^{0}(\cdot,\cdot)$.
From the approximation theory we know
\begin{equation}\label{capproximation}
\|\boldsymbol{v}-\Pi_{h}^{1}\boldsymbol{v}\|_{T}+
h_{T}|\boldsymbol{v}-\Pi_{h}^{1}\boldsymbol{v}|_{1,T}\leq {C_{I_{1}} h_{T}^{2}|\boldsymbol{v}|_{2,T}}\quad \forall~T\in\mathcal{T}_{h},
\end{equation}
which, together with \cref{L2approximation0}, gives
\begin{equation}\label{rt0part}
\|\Pi_{h}^{R}(\boldsymbol{v}-\Pi_{h}^{1}\boldsymbol{v})\|_{T}\leq {C_{R}C_{I_{1}} h_{T}^{2}|\boldsymbol{v}|_{2,T}.}
\end{equation}
Then \cref{approximation1} follows immediately from \cref{capproximation} and \cref{rt0part}, together with the triangle inequality.
\end{proof}
\subsection{Stability and boundedness}
\label{subsec:33}
From \cref{norm}, the coercivity and the boundedness of $a_{h}(\cdot,\cdot)$ are trivial. We just {state} the following lemmas {without proof}.
\begin{lemma}[Coercivity]
For any $\boldsymbol{v} \in V(h)$ we have
\begin{equation}\label{ellipticity}
a_{h}(\boldsymbol{v}, \boldsymbol{v}) = |||\boldsymbol{v}|||^{2}.
\end{equation}
\end{lemma}

\begin{lemma}[Boundedness]
For any $\boldsymbol{w},\boldsymbol{v} \in V(h)$ we have
\begin{equation}\label{boundedness}
a_{h}(\boldsymbol{v}, \boldsymbol{w}) \leq|||\boldsymbol{v}||||||\boldsymbol{w}|||.
\end{equation}
\end{lemma}

Next, let us prove the discrete inf-sup condition here. Similarly to the Bernardi and Raugel element, here we can not use the interpolation operator $\Pi_{h}$ to prove the inf-sup conditions since it is not well-defined on $[H^{1}(\Omega)]^{d}$. However, using the similar technique in \cite{bernardi_analysis_1985} and \cref{interpolationlemma}, we can construct an interpolation $\mathcal{\pi}_{h}: V\rightarrow V_{h}$ which satisfies
\begin{equation}\label{commuting2}
\nabla\cdot\mathcal{\pi}_{h}\boldsymbol{v}=P_{h}\nabla\cdot\boldsymbol{v}\quad \forall~ \boldsymbol{v}\in V,
\end{equation}
and
\begin{equation}\label{LBBprepare2}
|||\mathcal{\pi}_{h}\boldsymbol{v}|||\leq C_{F} ||\boldsymbol{v}||_{1}\quad \forall~\boldsymbol{v}\in V.
\end{equation}
For more details, we refer the readers to \cite{bernardi_analysis_1985} and \cite[pp. 136-138]{girault_finite_1986}.
Based on this interpolation, we will prove the following inf-sup conditions.
\begin{lemma}[LBB Stability]
There exists a positive constant $\beta_{1},$ independent of $h,$ such that
\begin{equation}\label{inf-sup}
\sup _{\boldsymbol{v}_{h} \in V_{h}\setminus\{\boldsymbol{0}\}} \frac{b(\boldsymbol{v}_{h}, q_{h})}{|||\boldsymbol{v}_{h}|||} \geq \beta_{1}\|q_{h}\| \quad \forall~ q_{h} \in W_{h}.
\end{equation}
\end{lemma}
\begin{proof}
We first use the operator $\mathcal{\pi}_{h}$ to obtain
\begin{equation}\label{inequality1}
\sup _{\boldsymbol{v}_{h} \in V_{h}\setminus\{\boldsymbol{0}\}} \frac{b(\boldsymbol{v}_{h}, q_{h})}{|||\boldsymbol{v}_{h}|||} \geq \sup _{\boldsymbol{v} \in V\setminus\{\boldsymbol{0}\}} \frac{b\left(\mathcal{\pi}_{h} \boldsymbol{v}, q_{h}\right)}{|||\mathcal{\pi}_{h} \boldsymbol{v}|||}=\sup _{\boldsymbol{v} \in V\setminus\{\boldsymbol{0}\}} \frac{b(\boldsymbol{v}, q_{h})}{|||\mathcal{\pi}_{h} \boldsymbol{v}|||}.
\end{equation}

Thus, substituting \cref{LBBprepare2} into the inequality \cref{inequality1} gives
\begin{displaymath}
\sup _{\boldsymbol{v}_{h} \in V_{h}\setminus\{\boldsymbol{0}\}} \frac{b(\boldsymbol{v}_{h}, q_{h})}{|||\boldsymbol{v}_{h}|||} \geq C_{F}^{-1} \sup _{\boldsymbol{v} \in V\setminus\{\boldsymbol{0}\}} \frac{b(\boldsymbol{v}, q_{h})}{\|\boldsymbol{v}\|_{1}} \geq \beta_{1}\|q_{h}\|,
\end{displaymath}
where we have used the inf-sup condition for the continuous case \cref{continuousInfSup}.
\end{proof}
\subsection{Consistency Error}
\label{subsec:34}
The continuous solutions $(\boldsymbol{u},p)$ satisfy
\begin{equation}\label{fullsatisfy}
\begin{split}
\nu \left(-\Delta\boldsymbol{u},\boldsymbol{v}\right)-b\left(\boldsymbol{v},p\right)&
=\left(\boldsymbol{f},\boldsymbol{v}\right)\quad\forall~\boldsymbol{v}\in V(h),\\
b\left(\boldsymbol{u},q\right)&=0 \quad\forall~q\in W,
\end{split}
\end{equation}
which can be rewritten as
\begin{equation}\label{consistencyanalysis0}
\begin{split}
\nu a_{h}(\boldsymbol{u},\boldsymbol{v})-b(\boldsymbol{v},p)&=(\boldsymbol{f},\boldsymbol{v})-\nu
\left((-\Delta\boldsymbol{u},\boldsymbol{v})-a_{h}(\boldsymbol{u},\boldsymbol{v})\right)
\quad\forall~\boldsymbol{v}\in V(h),\\
b(\boldsymbol{u},q)&=0 \quad\forall~q\in W.
\end{split}
\end{equation}
Then $\left(-\Delta\boldsymbol{u},\boldsymbol{v})-a_{h}(\boldsymbol{u},\boldsymbol{v}\right)$ {is the consistency error} of our method.
From the arguments in \cite[Sect. 3]{wang_new_2007}, one can see, in general,
{$\left|(-\Delta\boldsymbol{u},\boldsymbol{v})-a_{h}\left(\boldsymbol{u}, \boldsymbol{v}\right)\right|\neq 0.$}
\begin{lemma}[Consistency]The following inequality holds:
\begin{equation}\label{consistencyestimate}
\left|(-\Delta\boldsymbol{u},\boldsymbol{v})-a_{h}\left(\boldsymbol{u}, \boldsymbol{v}\right)\right| \leq (\sum_{T\in\mathcal{T}_{h}}h_{T}^{2}\|\Delta\boldsymbol{u}\|_{T}^{2})^{1/2}(\sum_{T\in\mathcal{T}_{h}}h_{T}^{-2}\left\|\boldsymbol{v}^{R}\right\|_{T}^{2})^{1/2} \quad \forall~\boldsymbol{v} \in {{V}(h)}.
\end{equation}
\end{lemma}
\begin{proof}
Note that $a_{h}(\boldsymbol{u},\boldsymbol{v}_{h})=a(\boldsymbol{u},\boldsymbol{v}^{1})=(-\Delta\boldsymbol{u},\boldsymbol{v}^{1})$. Thus, we have
\begin{subequations}\nonumber
\begin{align}
\left|(-\Delta\boldsymbol{u},\boldsymbol{v})-a_{h}\left(\boldsymbol{u}, \boldsymbol{v}\right)\right|&=\left|(-\Delta \boldsymbol{u},\boldsymbol{v}^{R})\right|\leq \sum_{T\in\mathcal{T}_{h}} h_{T}\|\Delta\boldsymbol{u}\|_{T}h_{T}^{-1}\left\|\boldsymbol{v}^{R}\right\|_{T}\\&\leq (\sum_{T\in\mathcal{T}_{h}}h_{T}^{2}\|\Delta\boldsymbol{u}\|_{T}^{2})^{1/2}(\sum_{T\in\mathcal{T}_{h}}h_{T}^{-2}\left\|\boldsymbol{v}^{R}\right\|_{T}^{2})^{1/2}.
\end{align}
\end{subequations}
\end{proof}
\subsection{Error estimates}
\label{subsec:35}
Based on the above results, the error analysis can be done now. In addition, the following error equations are needed:
\begin{subequations}\label{erroreqs}\begin{align}
\nu a_{h}\left(\boldsymbol{u}-\boldsymbol{u}_{h}, \boldsymbol{v}_{h}\right)-b\left(\boldsymbol{v}_{h}, p-p_{h}\right)&=-\nu\delta_{h} (\boldsymbol{u},\boldsymbol{v}_{h}) \quad\forall~ \boldsymbol{v}_{h} \in V_{h}, \label{erroreqs1}\\
b\left(\boldsymbol{u}-\boldsymbol{u}_{h}, q_{h}\right)&=0 \quad\forall~ q_{h} \in W_{h},\label{erroreqs2}
\end{align}\end{subequations}
where $\delta_{h}:V(h)\times V(h)\rightarrow \mathbb{R}$ is the consistency error defined as
\begin{displaymath}
\delta_{h}(\boldsymbol{u},\boldsymbol{v})=(-\Delta \boldsymbol{u},\boldsymbol{v})-a_{h}\left(\boldsymbol{u}, \boldsymbol{v}\right).
\end{displaymath}
These error equations can be obtained by subtracting \cref{formula2} from \cref{consistencyanalysis0}.
In what follows we shall use $C_{\alpha}$ to denote the constant in the upper bound of the consistency error such that
\begin{displaymath}
|\delta_{h}(\boldsymbol{u},\boldsymbol{v})|\leq C_{\alpha}h||\Delta\boldsymbol{u}||\left\{a^{R}(\boldsymbol{v}^{R},\boldsymbol{v}^{R})\right\}^{1/2}.
\end{displaymath}
It follows from \cref{consistencyestimate} that $C_{\alpha}$ only depends on the parameters $\alpha_{T}$ when $a^{R}=a^{0}$.
Now we are in a position to present an error estimate for the finite element scheme \cref{formula2}. {Our attention is} mainly paid to the velocity estimates, which are independent of the pressure and the viscosity.
\begin{theorem}\label{theorem1}
 Let $(\boldsymbol{u} , p)$ be the solution of \cref{Stokes} and $\left(\boldsymbol{u}_{h} , p_{h}\right) \in V_{h} \times W_{h}$ be the solution of \cref{formula2}. Assume that $(\boldsymbol{u},p)\in [H^{2}(\Omega)]^{d}\times H^{1}(\Omega)$. Then there exists a constant $C$ independent of $h$, $\nu$ and $\beta_{1}$ such that
\begin{equation}\label{estimate}
|||\boldsymbol{u}-\boldsymbol{u}_{h}|||\leq C h |\boldsymbol{u}|_{2},\quad \|p_{h}-P_{h}p\|\leq \frac{\nu}{\beta_{1}}C h |\boldsymbol{u}|_{2},
\end{equation}
and
\begin{equation}\label{pressureestimate}
\left\|p-p_{h}\right\| \leq C h \left(\frac{\nu}{\beta_{1}}|\boldsymbol{u}|_{2}+\left|p\right|_{1}\right).
\end{equation}
Here $\beta_{1}$ is the coefficient of the discrete inf-sup condition.
\end{theorem}
\begin{proof}
Introduce the {notation}:
$$
\xi_{h}=\boldsymbol{u}_{h}-\Pi_{h} \boldsymbol{u}, \quad \eta_{h}=p_{h}-P_{h} p,\quad
\xi=\boldsymbol{u}-\Pi_{h} \boldsymbol{u}, \quad \eta=p-P_{h} p.
$$
We note that \cref{commuting} implies  $\nabla\cdot\Pi_{h}\boldsymbol{u}=0$, which, together with \cref{divergencefree} gives that
\begin{equation}\label{divergencefree1}
\nabla\cdot\xi_{h}=0.
\end{equation}
It follows from the error equations \cref{erroreqs} that
\begin{subequations}\label{erroreq}
\begin{align}
\nu a_{h}\left(\xi_{h}, \boldsymbol{v}_{h}\right)-b\left(\boldsymbol{v}_{h}, \eta_{h}\right)&=\nu a_{h}(\xi, \boldsymbol{v}_{h})-b(\boldsymbol{v}_{h}, \eta)+\nu\delta_{h}(\boldsymbol{u},\boldsymbol{v}_{h})\quad\forall~\boldsymbol{v}_{h}\in V_{h},\label{erroreq1}\\
b\left(\xi_{h}, q_{h}\right)&=b(\xi, q_{h})=0\quad\forall ~q_{h}\in W_{h}.\label{erroreq2}
\end{align}
\end{subequations}
By taking $\boldsymbol{v}_{h}=\xi_{h}$ in \cref{erroreq1} and $q_{h}=\eta_{h}$ in \cref{erroreq2}, the sum of \cref{erroreq1} and \cref{erroreq2} together with \cref{divergencefree1} gives
\begin{equation}\label{velocityerroreq}
a_{h}\left(\xi_{h}, \xi_{h}\right)=a_{h}\left(\xi, \xi_{h}\right)+\delta_{h}(\boldsymbol{u},\xi_{h}),
\end{equation}
where $\nu$ has been eliminated.
Thus, it follows from the coercivity \cref{ellipticity}, the boundedness \cref{boundedness} and the consistency error \cref{consistencyestimate} that
$$
|||\xi_{h}|||^{2} \leq |||\xi|||\times|||\xi_{h}|||+C_{\alpha} h|\boldsymbol{u}|_{2}|||\xi_{h}|||,
$$
which implies
$$
|||\xi_{h}||| \leq |||\xi|||+C_{\alpha} h|\boldsymbol{u}|_{2}.
$$
The above estimate is equivalent to
\begin{displaymath}
|||\boldsymbol{u}_{h}-\Pi_{h} \boldsymbol{u}||| \leq |||\boldsymbol{u}-\Pi_{h} \boldsymbol{u}|||+C_{\alpha} h|\boldsymbol{u}|_{2}.
\end{displaymath}
By taking $s=2$ in \cref{approximation1}, a combination of the triangle inequality and the above error estimate gives
\begin{equation}\label{errorofvelocity}
|||\boldsymbol{u}-\boldsymbol{u}_{h}||| \leq 2|||\boldsymbol{u}-\Pi_{h} \boldsymbol{u}|||+C_{\alpha} h|\boldsymbol{u}|_{2}\leq \left(2C_{I}+C_{\alpha}\right) h |\boldsymbol{u}|_{2}.
\end{equation}

Let us estimate the pressure error. Since $\nabla\cdot V_{h}=W_{h}$, it follows from \cref{L2orthogonal} that
$b(\boldsymbol{v}_{h},p-P_{h}p)=0$ for all $\boldsymbol{v}_{h}\in V_{h},$ which, together with the discrete inf-sup condition \cref{inf-sup} gives
$$
\begin{aligned}
\left\|p_{h}-P_{h} p\right\|&  \leq \frac{1}{\beta_{1}} \sup _{\boldsymbol{v}_{h} \in V_{h}\setminus\{\boldsymbol{0}\}} \frac{b\left(\boldsymbol{v}_{h}, p_{h}-P_{h} p\right)}{|||\boldsymbol{v}_{h}|||}
=\frac{1}{\beta_{1}} \sup _{\boldsymbol{v}_{h} \in V_{h}\setminus\{\boldsymbol{0}\}} \frac{b\left(\boldsymbol{v}_{h}, p_{h}-p\right)}{|||\boldsymbol{v}_{h}|||}\\
&=\frac{1}{\beta_{1}} \sup _{\boldsymbol{v}_{h} \in V_{h}\setminus\{\boldsymbol{0}\}} \frac{-\nu a_{h}\left(\boldsymbol{u}-\boldsymbol{u}_{h}, \boldsymbol{v}_{h}\right)-\nu\delta_{h}(\boldsymbol{u},\boldsymbol{v}_{h})}{|||\boldsymbol{v}_{h}|||} \\
 &\leq \frac{\nu}{\beta_{1}}\sup _{\boldsymbol{v}_{h} \in V_{h}\setminus\{\boldsymbol{0}\}} \frac{1}{|||\boldsymbol{v}_{h}|||}|||\boldsymbol{v}_{h}|||\left(|||\boldsymbol{u}-\boldsymbol{u}_{h}|||
 +C_{\alpha}h|\boldsymbol{u}|_{2}\right)
 \leq \frac{\nu}{\beta_{1}}\left(|||\boldsymbol{u}-\boldsymbol{u}_{h}|||+C_{\alpha}h|\boldsymbol{u}|_{2}\right).
\end{aligned}
$$
Now using the above estimate and the estimate \cref{errorofvelocity} we get
$$
\left\|p_{h}-P_{h} p\right\|\leq 2\frac{\nu}{\beta_{1}}\left(C_{I}+C_{\alpha}\right) h |\boldsymbol{u}|_{2}.
$$
Thus we complete the proof of \cref{estimate}.
Then the error estimate \cref{pressureestimate} follows from the standard triangle inequality, the above estimate and taking $m=1$ in \cref{L2projectionerror}.
\end{proof}
\begin{remark}
Classical mixed methods which do not satisfy divergence-free (or pressure-robust) property can only obtain an abstract estimate as
\begin{displaymath}
\left\|\nabla\left(\boldsymbol{u}-\boldsymbol{u}_{h}^{cla}\right)\right\|_{L^{2}(\Omega)} \leq C \left(\inf _{{\boldsymbol{v}}_{h} \in V_{h}^{cla}}\left\|\nabla\left(\boldsymbol{u}-{\boldsymbol{v}}_{h}\right)\right\|_{L^{2}(\Omega)}+{\nu}^{-1} \inf _{q_{h} \in W_{h}^{cla}}\left\|p-q_{h}\right\|_{L^{2}(\Omega)}\right).
\end{displaymath}
See \cite[eq. 3.5]{john_divergence_2017}. In classical mixed methods, when $\nu$ is small, or when the best approximation error of the pressure is large, the error bound of the velocity may be large.
\end{remark}

To derive an optimal order $L^{2}$-error estimate for the velocity approximation, we seek $(\boldsymbol{z} , \lambda) $ satisfying
{$$
-\Delta \boldsymbol{z}+\nabla \lambda=\boldsymbol{u}-\boldsymbol{u}_{h}, \
\nabla \cdot \boldsymbol{z}=0\text { in } \Omega,
\quad\boldsymbol{z}=\boldsymbol{0}\text { on } \partial \Omega,
$$}
which is also a Stokes system with the viscosity equal to 1.

Here we assume that $\Omega$ is convex, and since $\boldsymbol{u}-\boldsymbol{u}_{h}\in \left[L^{2}\left(\Omega\right)\right]^{d}$, we have $(\boldsymbol{z},\lambda)\in [H^{2}(\Omega)]^{d}\times H^{1}(\Omega)$ and the following regularity estimate holds true (see \cite{Dauge1989,Twogrid2001}):
\begin{equation}\label{regularityestimate}
\|\boldsymbol{z}\|_{2}+\|\lambda\|_{1} \leq C\left\|\boldsymbol{u}-\boldsymbol{u}_{h}\right\|.
\end{equation}

Define $\delta_{h}^{*}:V(h)\times V(h)\rightarrow \mathbb{R}$ as
{$\delta_{h}^{*}(\boldsymbol{v},\boldsymbol{z})=\delta_{h}(\boldsymbol{z},\boldsymbol{v}),$}
which is the adjoint consistency error \cite{arnold_unified_2002}. In fact, since $-\Delta$ is a self-adjoint operator and $a_{h}(\cdot,\cdot)$ is symmetric, the properties of consistency and adjoint consistency are equivalent.

Similarly to \cref{consistencyanalysis0}, note that $a_{h}(\cdot,\cdot)$ is symmetric, then we have
\begin{equation}\label{erroreq3}
a_{h}(\boldsymbol{v},\boldsymbol{z})-b(\boldsymbol{v}, \lambda)=\left(\boldsymbol{u}-\boldsymbol{u}_{h}, \boldsymbol{v}\right)-\delta_{h}^{*}(\boldsymbol{v},\boldsymbol{z})\quad\forall~\boldsymbol{v}\in V(h).
\end{equation}

\begin{theorem}\label{theorem2}
Let $(\boldsymbol{u} , p)$ be the solution of \cref{Stokes} and $\left(\boldsymbol{u}_{h} , p_{h}\right) \in V_{h} \times W_{h}$ be the solution of \cref{formula2}. Assume that $(\boldsymbol{u},p)\in [H^{2}(\Omega)]^{d}\times H^{1}(\Omega)$. Then there exists a constant $C$ independent of $h$, $\nu$ and $\beta_{1}$ such that
\begin{equation}\label{L2estimate}
\left\|\boldsymbol{u}-\boldsymbol{u}_{h}\right\| \leq C h^{2} |\boldsymbol{u}|_{2}.
\end{equation}
\end{theorem}
\begin{proof}
By taking $\boldsymbol{v}=\boldsymbol{u}-\boldsymbol{u}_{h}$ in \cref{erroreq3} we get
\begin{equation}\label{erroreq4}
 a_{h}\left(\boldsymbol{u}-\boldsymbol{u}_{h}, \boldsymbol{z}\right)-b\left(\boldsymbol{u}-\boldsymbol{u}_{h}, \lambda\right)=\left\|\boldsymbol{u}-\boldsymbol{u}_{h}\right\|^{2}-\delta_{h}^{*}(\boldsymbol{u}-\boldsymbol{u}_{h},\boldsymbol{z}).
\end{equation}
\Cref{erroreqs2} implies that
\begin{equation}\label{erroreq5}
b\left(\boldsymbol{u}-\boldsymbol{u}_{h}, \lambda\right)=b\left(\boldsymbol{u}-\boldsymbol{u}_{h}, \lambda-P_{h} \lambda\right).
\end{equation}
Meanwhile, \Cref{erroreqs1} and the fact $\nabla\cdot\Pi_{h}\boldsymbol{z}=0$ give
\begin{equation}\label{erroreq6}
\begin{aligned}
a_{h}\left(\boldsymbol{u}-\boldsymbol{u}_{h}, \boldsymbol{z}\right) &=a_{h}\left(\boldsymbol{u}-\boldsymbol{u}_{h}, \boldsymbol{z}-\Pi_{h} \boldsymbol{z}\right)+a_{h}\left(\boldsymbol{u}-\boldsymbol{u}_{h}, \Pi_{h} \boldsymbol{z}\right) \\
&=a_{h}\left(\boldsymbol{u}-\boldsymbol{u}_{h}, \boldsymbol{z}-\Pi_{h} \boldsymbol{z}\right)+\frac{1}{\nu}b\left(\Pi_{h} \boldsymbol{z}, p-p_{h}\right)-\delta_{h}(\boldsymbol{u},{\Pi_{h}\boldsymbol{z}}) \\
&=a_{h}\left(\boldsymbol{u}-\boldsymbol{u}_{h}, \boldsymbol{z}-\Pi_{h} \boldsymbol{z}\right)
-\delta_{h}(\boldsymbol{u},{\Pi_{h}\boldsymbol{z}}).
\end{aligned}
\end{equation}
Substituting \cref{erroreq5} and \cref{erroreq6} into \cref{erroreq4} one can obtain
\begin{displaymath}
\left\|\boldsymbol{u}-\boldsymbol{u}_{h}\right\|^{2}=a_{h}\left(\boldsymbol{u}-\boldsymbol{u}_{h}, \boldsymbol{z}-\Pi_{h} \boldsymbol{z}\right)-b\left(\boldsymbol{u}-\boldsymbol{u}_{h}, \lambda-P_{h} \lambda\right)
+\delta_{h}^{*}(\boldsymbol{u}-\boldsymbol{u}_{h},\boldsymbol{z})-\delta_{h}(\boldsymbol{u},{\Pi_{h}\boldsymbol{z}}).
\end{displaymath}
From \cref{interpolation} and \cref{consistencyestimate} we have
{\begin{displaymath}
|\delta_{h}(\boldsymbol{u},\Pi_{h}\boldsymbol{z})|\leq (\sum_{T\in\mathcal{T}_{h}}h_{T}^{2}\|\Delta\boldsymbol{u}\|_{T}^{2})^{1/2}(\sum_{T\in\mathcal{T}_{h}}h_{T}^{-2}\left\|\Pi_{h}^{R}(\boldsymbol{z}-\Pi_{h}^{1}\boldsymbol{z})\right\|_{T}^{2})^{1/2}
\leq C h^{2}|\boldsymbol{u}|_{2}||\boldsymbol{z}||_{2}\quad ({\rm see~ \cref{rt0part}}),
\end{displaymath}}
and {$|\delta_{h}^{*}(\boldsymbol{u}-\boldsymbol{u}_{h},\boldsymbol{z})|\leq C_{\alpha} h ||\boldsymbol{z}||_{2}|||\boldsymbol{u}-\boldsymbol{u}_{h}|||,$}
which imply
\begin{displaymath}
\left\|\boldsymbol{u}-\boldsymbol{u}_{h}\right\|^{2} \leq C |||\boldsymbol{u}-\boldsymbol{u}_{h}|||\left(|||\boldsymbol{z}-\Pi_{h} \boldsymbol{z}|||+\left\|\lambda-P_{h} \lambda\right\|\right)
+C h (h|\boldsymbol{u}|_{2}+|||\boldsymbol{u}-\boldsymbol{u}_{h}|||)||\boldsymbol{z}||_{2}.
\end{displaymath}
It follows from the above estimate, the interpolation error estimates \cref{L2projectionerror} and \cref{approximation1}, and the regularity estimate \cref{regularityestimate} that
$$
\left\|\boldsymbol{u}-\boldsymbol{u}_{h}\right\|^{2} \leq C h\left(|||\boldsymbol{u}-\boldsymbol{u}_{h}|||+h|\boldsymbol{u}|_{2}\right)\left\|\boldsymbol{u}-\boldsymbol{u}_{h}\right\|,
$$
which implies that
$$
\left\|\boldsymbol{u}-\boldsymbol{u}_{h}\right\| \leq C h\left(|||\boldsymbol{u}-\boldsymbol{u}_{h}|||+h|\boldsymbol{u}|_{2}\right).
$$
Then the estimate \cref{L2estimate} follows immediately from the above inequality and the estimate \cref{estimate}. Thus we complete the proof.
\end{proof}
\begin{remark}
From \cref{theorem1} and \cref{norm} we have $\|\nabla(\boldsymbol{u}-\boldsymbol{u}_{h}^{1})\|\leq C h |u|_{2}$. Thus, $\boldsymbol{u}_{h}^{1}$ is a conforming approximation of the continuous solution $\boldsymbol{u}$ with optimal order in $H^{1}$ norm. In addition, \cref{theorem1} also imply that
$||\boldsymbol{u}_{h}^{R}||\leq C h^{2}|\boldsymbol{u}|_{2}$, which, together with \cref{theorem2} and a triangle inequality, demonstrates that $\boldsymbol{u}_{h}^{1}$ is also convergent to $\boldsymbol{u}$ with the optimal order in $L^{2}$ norm.
\end{remark}

We summarize the above remark in the following corollary.
\begin{corollary}
Let $(\boldsymbol{u} , p)$ be the solution of \cref{Stokes} and $\left(\boldsymbol{u}_{h} , p_{h}\right) \in V_{h} \times W_{h}$ be the solution of \cref{formula2}. Assume that $(\boldsymbol{u},p)\in [H^{2}(\Omega)]^{d}\times H^{1}(\Omega)$. Then there exists a constant $C$ independent of $h$, $\nu$ and $\beta_{1}$ such that
\begin{equation}\nonumber
\left\|\boldsymbol{u}-\boldsymbol{u}_{h}^{1}\right\|+h\left\|\nabla(\boldsymbol{u}-\boldsymbol{u}_{h}^{1})\right\|\leq C h^{2} |\boldsymbol{u}|_{2}.
\end{equation}
\end{corollary}

\begin{remark}
Although our method is stable and optimally convergent as long as the parameters $\alpha_{T},T\in\mathcal{T}_{h}$ are positive, here we can not give a set of optimal choices in the sense of the accuracy. This is still an open and interesting topic. Consider the simplest case. Assuming that $\alpha_{T}\equiv \alpha$ and taking $a^{R}=a^{0}$ may give some rough answers for this topic. Note that \cref{velocityerroreq}, \cref{consistencyestimate} and the definition of $|||\cdot|||$ imply
{\begin{displaymath}
|||\xi_{h}|||\leq \left\{||\nabla(\boldsymbol{u}-\Pi_{h}^{1}\boldsymbol{u})||^{2}+\alpha \sum_{T\in\mathcal{T}_{h}}h_{T}^{-2}||\Pi_{h}^{R}(\boldsymbol{u}-\Pi_{h}^{1}\boldsymbol{u})||_{T}^{2}\right\}^{1/2}+
\left\{\alpha^{-1}\sum_{T\in\mathcal{T}_{h}}h_{T}^{2}||\Delta\boldsymbol{u}||_{T}^{2}\right\}^{1/2}.
\end{displaymath}}
Since $||\Pi_{h}^{R}(\boldsymbol{u}-\Pi_{h}^{1}\boldsymbol{u})||_{T}\leq C_{R}C_{I_{1}}h_{T}^{2}|\boldsymbol{u}|_{2,T}$ (see \cref{rt0part}) and $||\Delta\boldsymbol{u}||_{T}^{2}\leq d|\boldsymbol{u}|_{2,T}^{2}$, by a series of simple calculations we arrive at
\begin{equation}\nonumber
|||\xi_{h}|||^{2}\leq 2 ||\nabla(\boldsymbol{u}-\Pi_{h}^{1}\boldsymbol{u})||^{2}+2(\alpha C_{R}^{2}C_{I_{1}}^{2}+\alpha^{-1}d) \sum_{T\in\mathcal{T}_{h}}h_{T}^{2}|\boldsymbol{u}|_{2,T}^{2}.
\end{equation}
The above inequality implies that taking $\alpha=\sqrt{d}/(C_{R}C_{I_{1}})$ might be a good choice. This requires a severe estimate of $C_{R}$ and $C_{I_{1}}$, which depends on the shape regularity of the mesh. In the numerical experiments we will study the performance of our method with different parameters on various meshes.
\end{remark}
\begin{remark}
In \cref{remark3} we point out that in practice we prefer the bilinear form $a^{\operatorname{div}}(\cdot,\cdot)$. However, in the numerical experiments we find that this bilinear form
is a little sensitive to the change of the parameters. To reduce the sensitivity, some strategies could be applied. For example, we can modify the bilinear form $a_{h}(\cdot,\cdot)$ to
\begin{equation}\label{modifiedform}
a_{h}^{m}(\boldsymbol{u},\boldsymbol{v})=a_{h}(\boldsymbol{u},\boldsymbol{v})-\sum_{T\in\mathcal{T}_{h}}\gamma_{T}(\nabla\cdot\boldsymbol{u}^{1},\nabla\cdot\boldsymbol{v}^{1})_{T},
\end{equation}
where $0\leq\gamma_{T}< \alpha_{T}/(d+1), T\in\mathcal{T}_{h}$ are the parameters. The analysis of $a_{h}^{m}(\cdot,\cdot)$ is very similar to $a_{h}(\cdot,\cdot)$. The main difference is that we can not prove the coercivity of $a_{h}^{m}(\cdot,\cdot)$ on the whole $V(h)$. However, we can prove that $a_{h}^{m}(\cdot,\cdot)$ is coercive on the subspace of $V(h)$
\begin{equation}\nonumber
Z(h)=\{\boldsymbol{v}\in V(h): b(\boldsymbol{v},q)=0 \quad \forall\ q  \in W\}=\{\boldsymbol{v}\in V(h): \nabla\cdot\boldsymbol{v}=0\}.
\end{equation}
Indeed, for any $\boldsymbol{v}\in Z(h)$, we have $\nabla\cdot\boldsymbol{v}^{1}=-\nabla\cdot\boldsymbol{v}^{R}$, which implies that
{\begin{subequations}\nonumber
\begin{align}
\gamma_{T}(\nabla\cdot\boldsymbol{v}^{1},\nabla\cdot\boldsymbol{v}^{1})_{T}&=\gamma_{T}(\nabla\cdot\boldsymbol{v}^{R},\nabla\cdot\boldsymbol{v}^{R})_{T}
\leq (d+1)\gamma_{T}\sum_{e\in\partial T\cap\mathcal{E}^{0}}v_{e}^{2}(\nabla\cdot\boldsymbol{\Phi}_{e},\nabla\cdot\boldsymbol{\Phi}_{e})_{T}\\
&<\alpha_{T}\sum_{e\in\partial T\cap\mathcal{E}^{0}}v_{e}^{2}(\nabla\cdot\boldsymbol{\Phi}_{e},\nabla\cdot\boldsymbol{\Phi}_{e})_{T}.
\end{align}
\end{subequations}}

Then the coercivity of $a_{h}^{m}(\cdot,\cdot)$ on $Z(h)$ follows immediately from the above inequality.
Note that $Z(h)$ include the finite element space
{$Z_{h}=\{\boldsymbol{v}_{h}\in V_{h}: b(\boldsymbol{v}_{h},q_{h})=0 \quad \forall\ q_{h}  \in W_{h}\}=\{\boldsymbol{v}_{h}\in V_{h}: \nabla\cdot\boldsymbol{v}_{h}=0\}.$}
In addition, the boundedness of the bilinear forms and the inf-sup conditions could be similarly obtained.
Then from the classical {theory on saddle point systems} \cite{Brezzi1974}, the existence and uniqueness of the solutions are guaranteed. The error estimate also follows similarly. The numerical experiments below (see \cref{EX:52}) demonstrate that this modified scheme does reduce the sensitivity.
\end{remark}
\section{Numerical experiments}
\label{sec:5}
We perform some numerical experiments on four examples. The examples in \cref{EX:51} and \cref{EX:53} are from \cite{john_divergence_2017}. In some examples, we also give some results computed by the classical low order Bernardi and Raugel pair \cite{bernardi_analysis_1985} to make a comparison.
\subsection{Convergence test}
\label{EX:51}
The exact solutions are prescribed as follows,
{\begin{displaymath}
\boldsymbol{u}=200\left(
x^{2}(1-x)^{2} y(1-y)(1-2 y),
-x(1-x)(1-2 x) y^{2}(1-y)^{2}
\right)^{\top},
\end{displaymath}}
and
\begin{displaymath}
p=10\left(\left(x-1/2\right)^{3} y^{2}+(1-x)^{3}\left(y-1/2\right)^{3}\right).
\end{displaymath}

The velocity field has the form of a large vortex \cite{john_divergence_2017}. We set $\Omega=(0,1)\times(0,1)$ and consider a small viscosity case, $\nu=10^{-6}$, to show the robustness of our methods. The corresponding results are shown in \cref{tab:1}, where we also give some results from the Bernardi and Raugel element. For the choice of parameters, we take $\alpha_{T}\equiv20, T \in \mathcal{T}_{h}$ for $a^{R}=a^{0}$ and $a^{R}=a^{D}$. In the case that $a^{R}=a^{\operatorname{div}}$, we take $\alpha_{T}\equiv1.5$. One could see, the velocity errors of our methods (``compact H(div) method") are roughly $10^{4}$ times smaller than the Bernardi and Raugel method. Moreover, comparing with the bottom line of \cref{tab:1}, {one can find that} the discrete pressures obtained by our methods are almost equal to the best approximation of the continuous pressure in $L^{2}(\Omega)$.
\newcommand{\minitab}[2][l]{\begin{tabular}{#1}#2\end{tabular}}
\begin{table}[htbp]
\caption{\cref{EX:51}. The $L^{2}$ and $H^{1}$ (seminorm) velocity errors and $L^{2}$ pressure errors. CHdiva0: the method \cref{formula2} with $a^{0}$ \cref{mathcalJ0}; CHdivaD: the method \cref{formula2} with $a^{D}$ \cref{mathcalJD}; CHdivadiv: the method \cref{formula2} with $a^{\operatorname{div}}$ \cref{mathcalJd}; B-R: the Bernardi and Raugel finite element method \cite{bernardi_analysis_1985}. The $L^{2}$ projection errors of the pressure are listed at the bottom.}
\label{tab:1}
\resizebox{\textwidth}{!}{
\begin{tabular}{cllllll}
\hline & & $h=0.1$ & $0.05$ & $0.025$ & $0.0125$ & $0.00625$ \\
\hline \multirow{4}{*} {$\left\|\boldsymbol{u}-\boldsymbol{u}_{h}\right\|$}
& ${\rm CHdiva0}$  & 1.51E-2 & 3.87E-3 & 9.73E-4 & 2.43E-4 & 6.09E-5 \\
& ${\rm CHdivaD}$  & 1.60E-2 & 4.11E-3 & 1.03E-3 & 2.60E-4 & 6.52E-5 \\
& ${\rm CHdivadiv}$  & 1.97E-2 & 5.05E-3 & 1.27E-3 & 3.19E-4 & 8.16E-5 \\
& ${\rm B-R}$  & 130.45 & 33.80 & 8.55 & 2.14 & 5.37E-1 \\
\hline \multirow{4}{*} {$\left\|\nabla(\boldsymbol{u}-\boldsymbol{u}_{h}^{1})\right\|$}
& ${\rm CHdiva0}$ & 1.12 & 5.67E-1 & 2.84E-1 & 1.42E-1 & 7.11E-2 \\
& ${\rm CHdivaD}$ & 1.13 & 5.68E-1 & 2.84E-1 & 1.42E-1 & 7.11E-2 \\
& ${\rm CHdivadiv}$ & 1.13 & 5.69E-1 & 2.84E-1 & 1.42E-1 & 7.11E-2  \\
& ${\rm B-R}$  & 7257.64 & 3895.73 & 2010.36 & 1020.13 & 513.69 \\
\hline \multirow{4}{*} {$\left\|p-p_{h}\right\|$}
& ${\rm CHdiva0}$ & 2.52E-2 & 1.26E-2 & 6.32E-3 & 3.16E-3 & 1.58E-3 \\
& ${\rm CHdivaD}$ & 2.52E-2 & 1.26E-2 & 6.32E-3 & 3.16E-3 & 1.58E-3 \\
& ${\rm CHdivadiv}$ & 2.52E-2 & 1.26E-2 & 6.32E-3 & 3.16E-3 & 1.58E-3 \\
& ${\rm B-R}$  & 2.61E-2 & 1.31E-2 & 6.55E-3 & 3.27E-3 & 1.63E-3 \\
\hline
{$\left\|p-P_{h}p\right\|$} &  & 2.52E-2 & 1.26E-2 & 6.32E-3 & 3.16E-3 & 1.58E-3 \\
\hline
\end{tabular}}
\end{table}
\subsection{Effects of different values of the parameters $\alpha_{T}$}
\label{EX:52}
In this example we focus on the case that $a^{R}=a^{\operatorname{div}}$ and investigate the effects of different choices of the parameters on four meshes (see \cref{mesh}). The exact solutions are the same as \cref{EX:51}. We set $\alpha_{T}\equiv\alpha$ with $\alpha=0.1:0.1:10$, that is, we test 100 different values of the parameters which are from 0.1 to 10.
The results from the original scheme \cref{formula2} and the modified scheme \cref{modifiedform} are shown in \cref{performance1} and \cref{performance2}, respectively. It is demonstrated that the most stable quantity is $||\boldsymbol{u}-\boldsymbol{u}_{h}^{1}||$ (\text{$L^{2}u^{1}$}), which is very close to the nodal interpolation error (\text{$L^{2}\Pi^{1}u$}) in the whole region. The quantities $||\boldsymbol{u}-\boldsymbol{u}_{h}||$ (\text{$L^{2}u$}) and $||\nabla(\boldsymbol{u}-\boldsymbol{u}_{h}^{1})||$ (\text{$H^{1}u^{1}$}) reduce violently before $\alpha<1$. All the {minimum points of $\|\boldsymbol{u}-\boldsymbol{u}_{h}\|$} are in $(1,5)$. The comparison between the results in \cref{performance1} and \cref{performance2} shows that the modified scheme \cref{modifiedform} does reduce the sensitivity to the change of $\alpha$ when $\alpha>1$.

Based on the above observations, we recommend that one choose $\alpha>1$ in practice.
\begin{figure}[htbp]
    \centering
    \begin{minipage}[t]{0.24\linewidth}
    \includegraphics[width=1\linewidth]{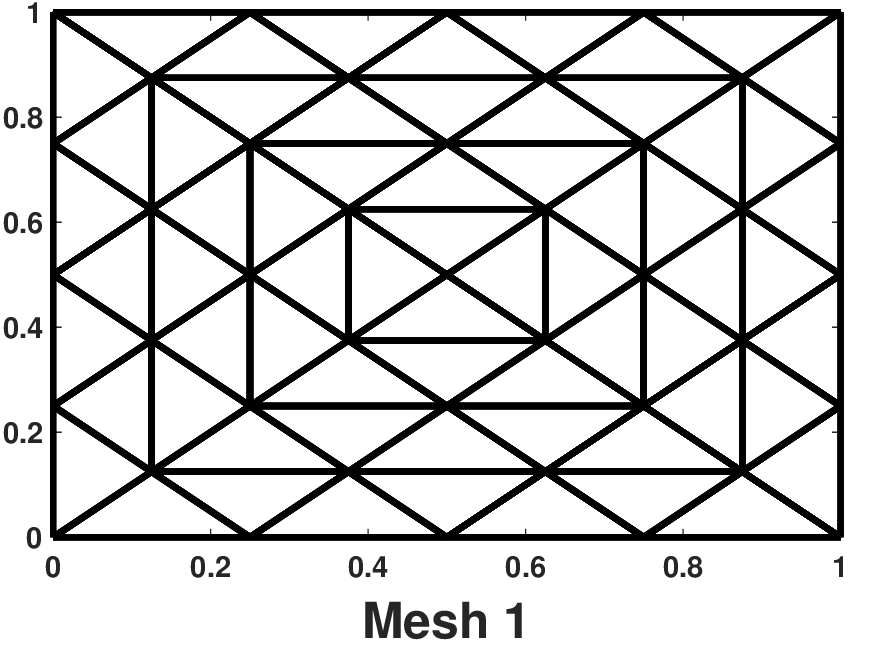}
    \end{minipage}
    \begin{minipage}[t]{0.24\linewidth}
    \includegraphics[width=1\linewidth]{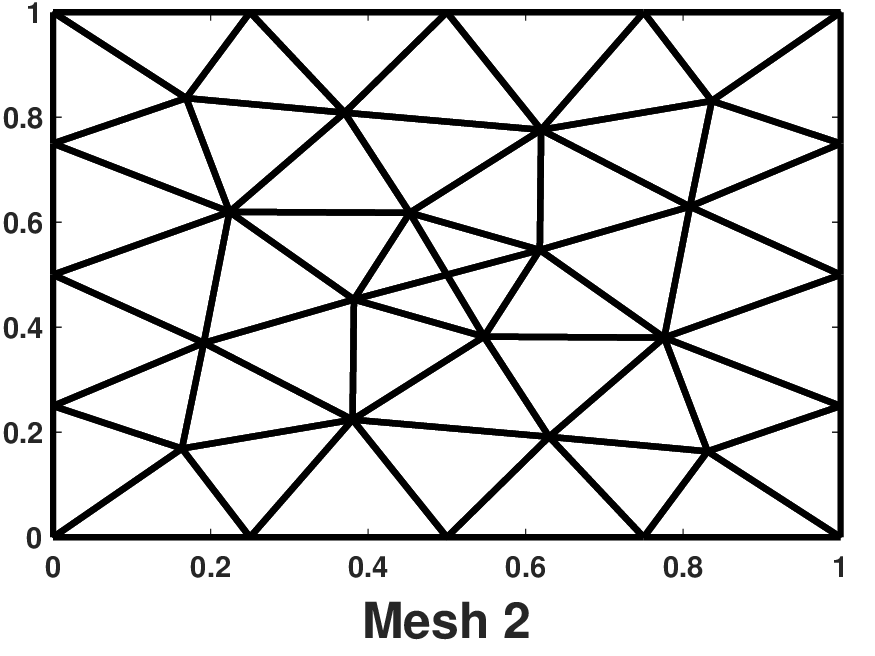}
    \end{minipage}
    \begin{minipage}[t]{0.24\linewidth}
    \includegraphics[width=1\linewidth]{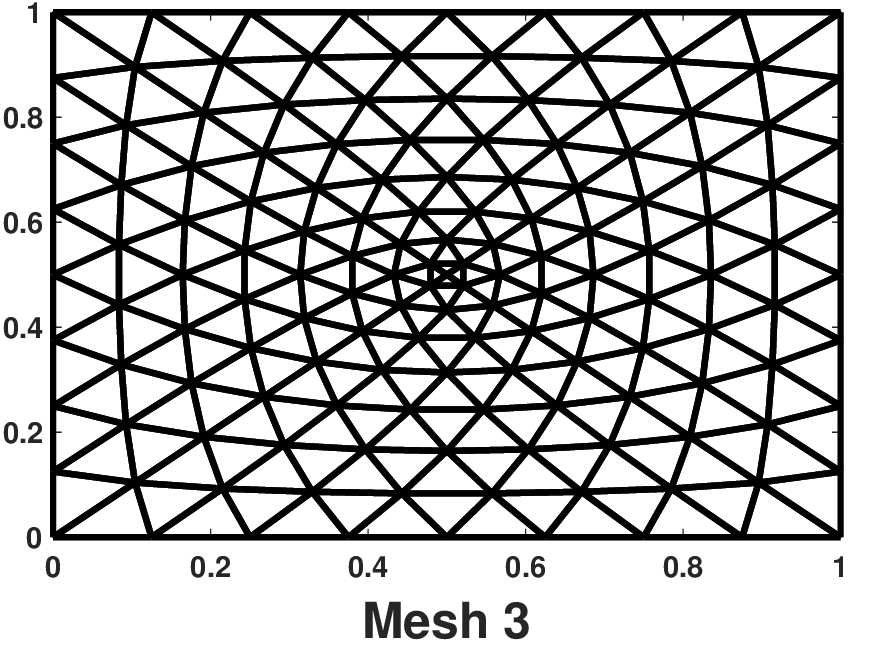}
    \end{minipage}
    \begin{minipage}[t]{0.24\linewidth}
    \includegraphics[width=1\linewidth]{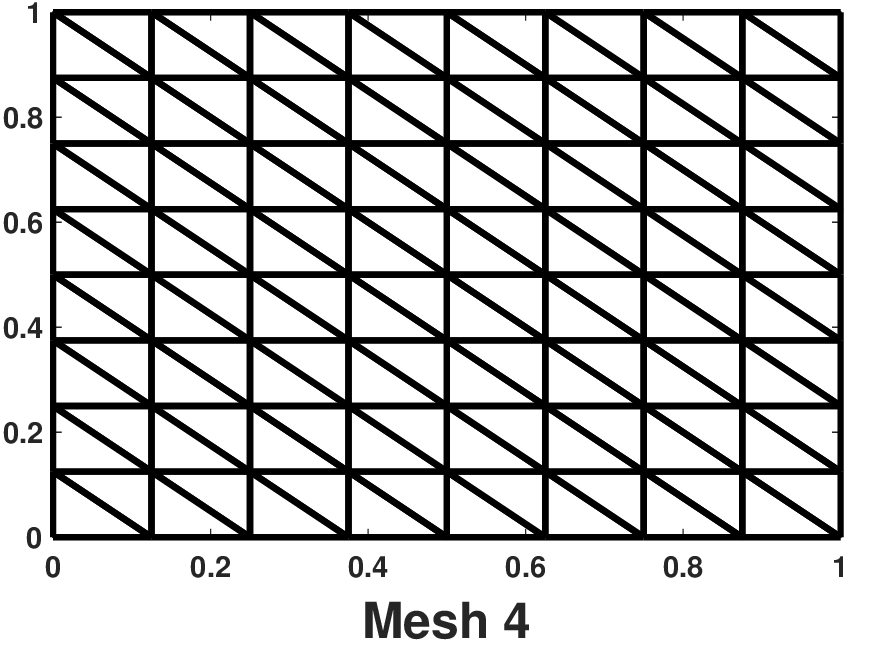}
    \end{minipage}
    \caption{\cref{EX:52}. Shown above are four test meshes.}
    \label{mesh}
\end{figure}
\begin{figure}[htbp]
    \centering
    \begin{minipage}[t]{0.24\linewidth}
    \includegraphics[width=1\linewidth]{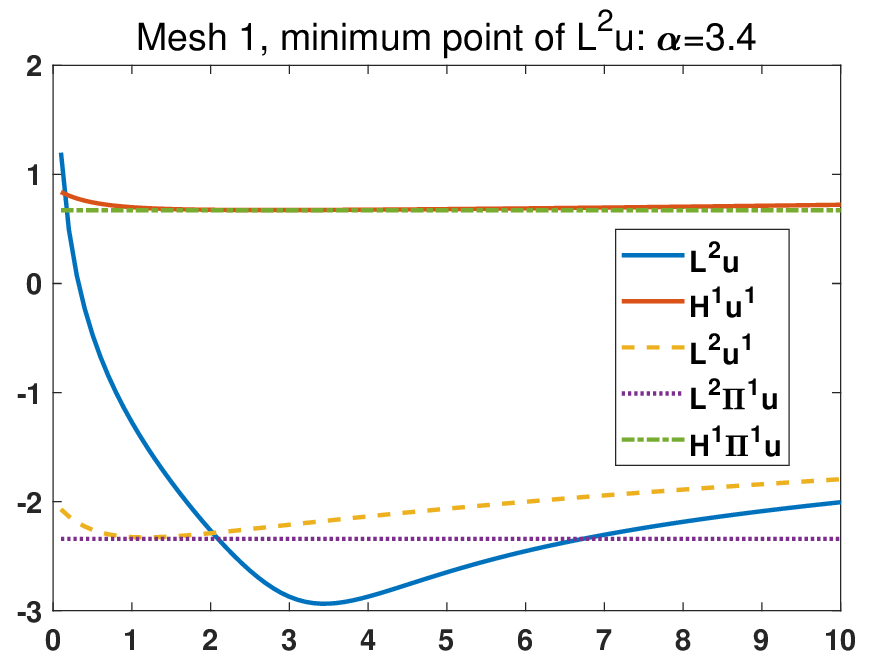}
    \end{minipage}
    \begin{minipage}[t]{0.24\linewidth}
    \includegraphics[width=1\linewidth]{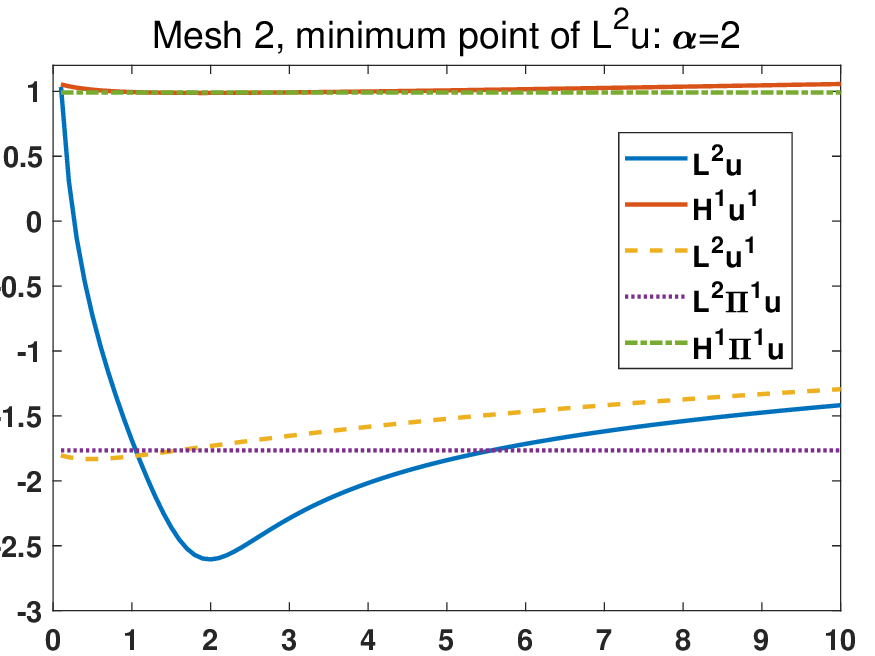}
    \end{minipage}
    \begin{minipage}[t]{0.24\linewidth}
    \includegraphics[width=1\linewidth]{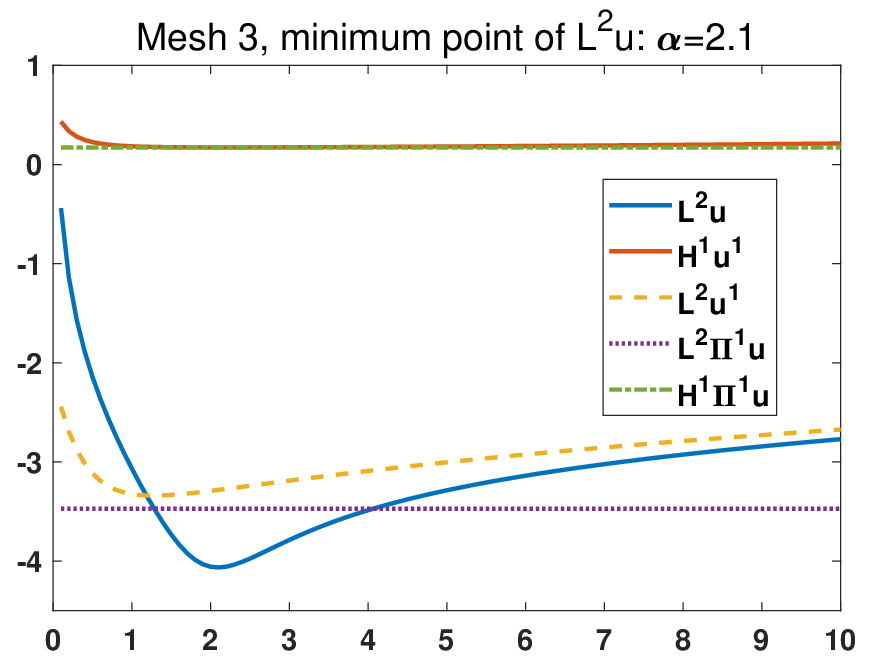}
    \end{minipage}
    \begin{minipage}[t]{0.24\linewidth}
    \includegraphics[width=1\linewidth]{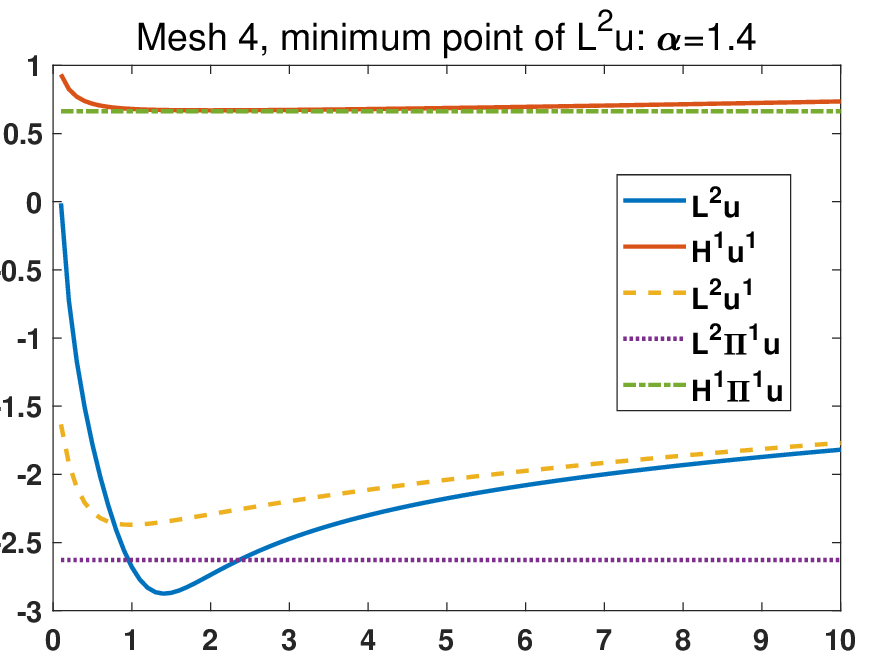}
    \end{minipage}
    \caption{\cref{EX:52}. {Shown above are plots of the parameter $\alpha$ versus the} {\it {natural logarithms}} of $||\boldsymbol{u}-\boldsymbol{u}_{h}||$ ($L^{2} u$), $||\nabla(\boldsymbol{u}-\boldsymbol{u}_{h}^{1})||$ ($H^{1} u^{1}$), $||\boldsymbol{u}-\boldsymbol{u}_{h}^{1}||$ ($L^{2} u^{1}$), $||\boldsymbol{u}-\Pi_{h}^{1}\boldsymbol{u}||$ ($L^{2} \Pi^{1}u$) and $||\nabla(\boldsymbol{u}-\Pi_{h}^{1}\boldsymbol{u})||$ ($H^{1}\Pi^{1} u$) on four meshes with the original scheme \cref{formula2}.}
    \label{performance1}
\end{figure}
\begin{figure}[htbp]
    \centering
    \begin{minipage}[t]{0.24\linewidth}
    \includegraphics[width=1\linewidth]{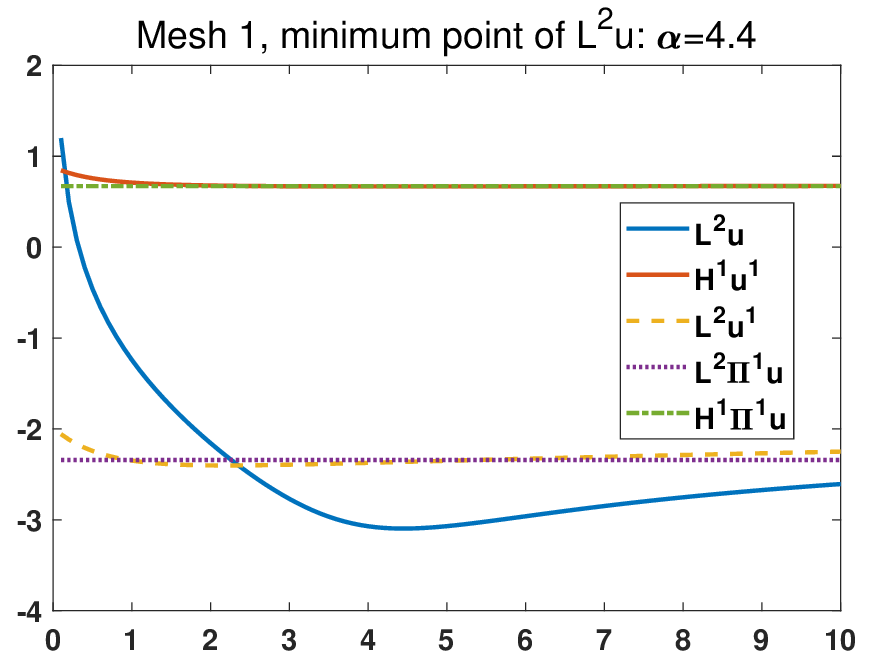}
    \end{minipage}
    \begin{minipage}[t]{0.24\linewidth}
    \includegraphics[width=1\linewidth]{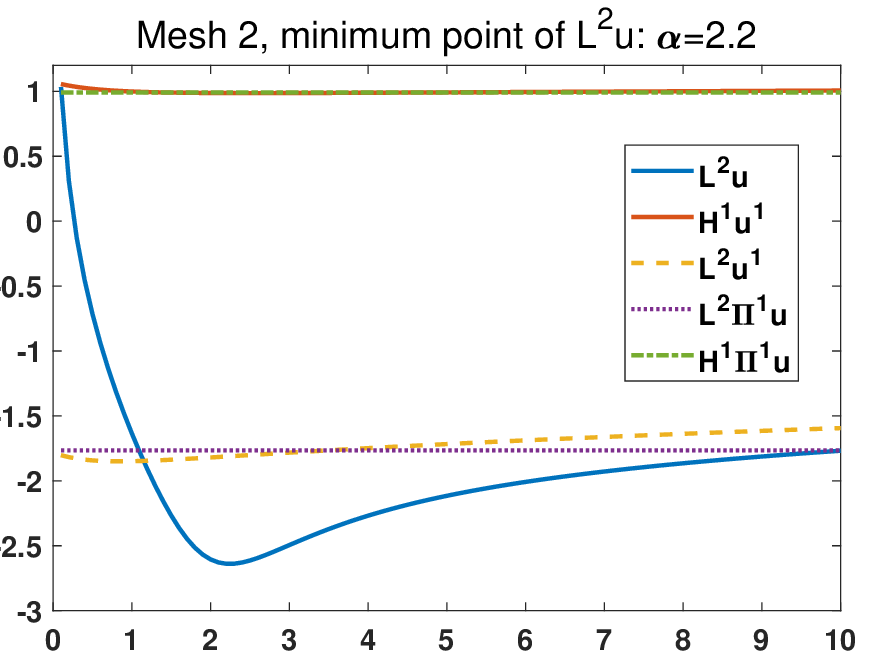}
    \end{minipage}
    \begin{minipage}[t]{0.24\linewidth}
    \includegraphics[width=1\linewidth]{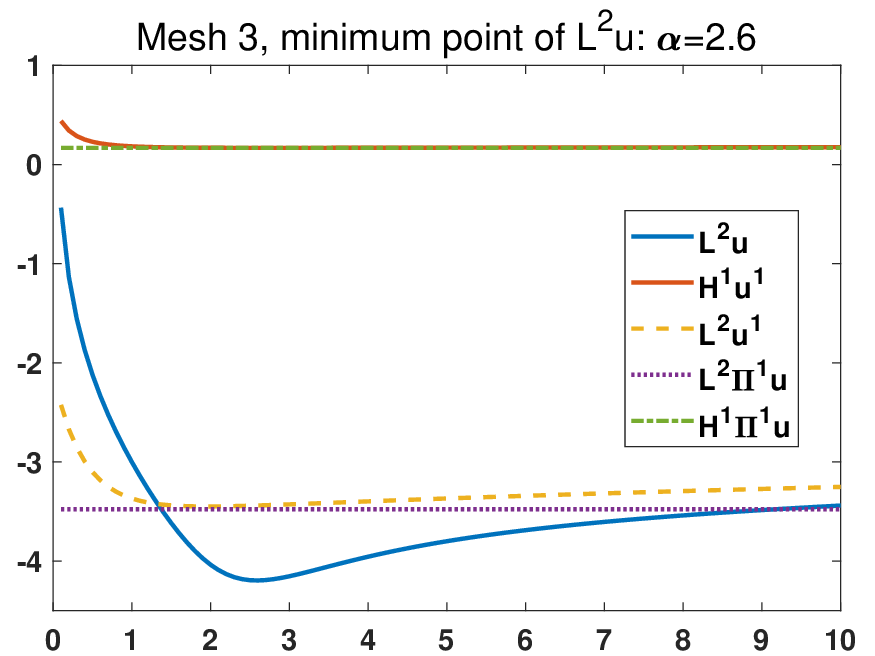}
    \end{minipage}
    \begin{minipage}[t]{0.24\linewidth}
    \includegraphics[width=1\linewidth]{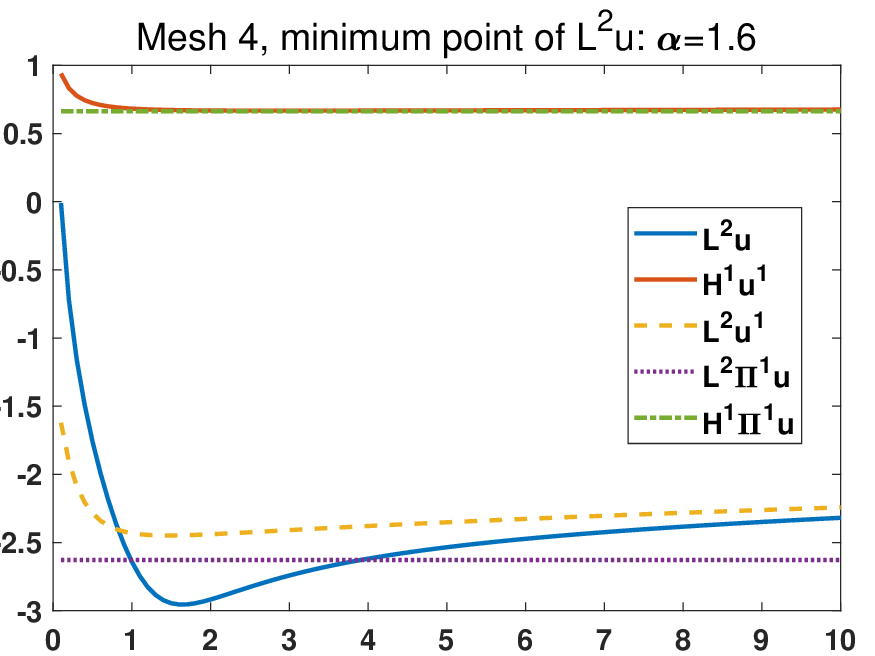}
    \end{minipage}
    \caption{\cref{EX:52}. {Shown above are plots of the parameter $\alpha$ versus the} {\it {natural logarithms}} of $||\boldsymbol{u}-\boldsymbol{u}_{h}||$ ($L^{2} u$), $||\nabla(\boldsymbol{u}-\boldsymbol{u}_{h}^{1})||$ ($H^{1} u^{1}$), $||\boldsymbol{u}-\boldsymbol{u}_{h}^{1}||$ ($L^{2} u^{1}$), $||\boldsymbol{u}-\Pi_{h}^{1}\boldsymbol{u}||$ ($L^{2} \Pi^{1}u$) and $||\nabla(\boldsymbol{u}-\Pi_{h}^{1}\boldsymbol{u})||$ ($H^{1}\Pi^{1} u$) on four meshes with the modified scheme \cref{modifiedform}.}
    \label{performance2}
\end{figure}

\subsection{Stokes flow with Coriolis force}
\label{EX:53}
In this example, we consider the flows with strong Coriolis forces which appear in several real-world applications such as meteorology. The simplest model is described as follows (see \cite{linke_velocity_2016,john_divergence_2017}),
\begin{equation}\label{Coriolis}
-\nu\Delta\boldsymbol{u}+\nabla p+2\boldsymbol{\omega}\times \boldsymbol{u}=\boldsymbol{f},\quad\nabla\cdot\boldsymbol{u}=0 \quad \text{in}\ \Omega,
\end{equation}
where $\boldsymbol{\omega}$ is a constant angular velocity vector.

Here $\boldsymbol{u}=(u_{1},u_{2})^{T}$. We embed it into three-dimensional space as $\boldsymbol{u}=(u_{1},u_{2},0)^{T}$, and similarly for other variables. We take $\boldsymbol{\omega}=(0,0,\omega)^{T}$ then we get a two-dimensional model since
{$\boldsymbol{\omega}\times \boldsymbol{u}=\omega(-u_{2},u_{1}, 0)^{T}.$}
The computational domain $\Omega$ is taken as a L-shaped domain: $(0,4)\times(0,2)\setminus[2,4)\times(0,1]$. We choose $\nu=1$. And the value of $\omega$ changes from $1$, $100$ to $10000$. We note that changing the magnitude $\omega$ of the Coriolis force will change only the pressure, and not the velocity (see \cite[Example 6.5]{john_divergence_2017}). The inlet is set at $x=0$ and outlet is set at $x=4$. Pure Dirichlet boundary condition is considered, where the volume preserving
parabolic inflow and outflow profiles are given by
{\begin{equation}\nonumber
\boldsymbol{u}_{in}=\left(
 y(2-y)/2,
0
\right)^{\top},
\quad\boldsymbol{u}_{out}=\left(
 4(2-y)(y-1),
0
\right)^{\top},
\end{equation}}
and homogeneous conditions are used in other parts of the boundary. Here the Coriolis term is discretized by $(2\boldsymbol{\omega}\times \boldsymbol{u}_{h},\boldsymbol{v}_{h})$, and the nonhomogeneous boundary condition is imposed by enforcing $\boldsymbol{u}_{h}|_{\Gamma}=\Pi_{h}\boldsymbol{u}|_{\Gamma}$. Of course, we do not know the exact solution $\boldsymbol{u}$ but the boundary value of $\Pi_{h}\boldsymbol{u}$ can be obtained by the known boundary value of $\boldsymbol{u}$. The computational results are shown in \cref{fig54}.

\begin{figure}
 \centering
 \subfigure{
    \begin{minipage}[b]{0.23\linewidth}
    \includegraphics[width=1\linewidth]{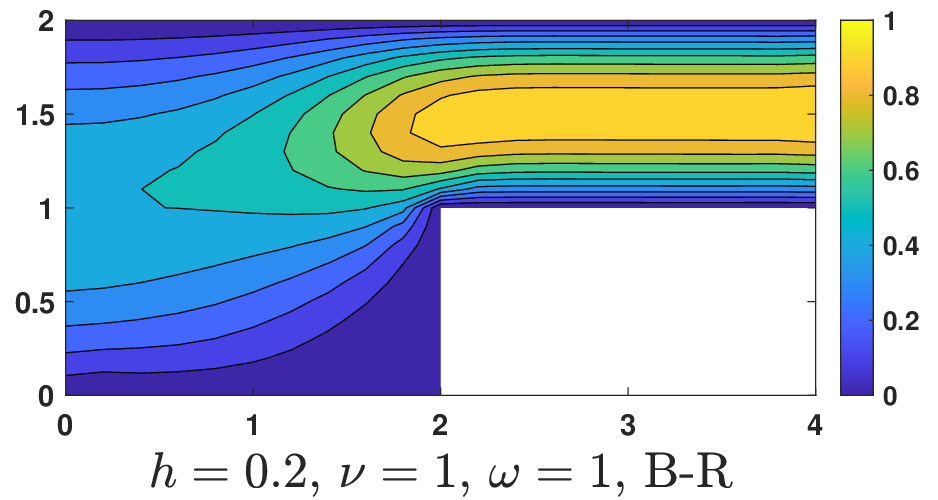}\vspace{4pt}
    \includegraphics[width=1\linewidth]{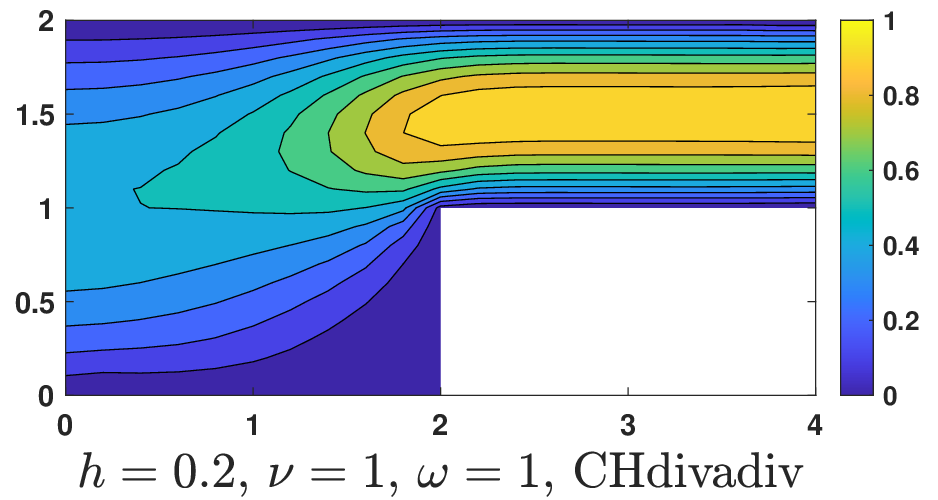}
    \end{minipage}}
 \subfigure{
    \begin{minipage}[b]{0.23\linewidth}
    \includegraphics[width=1\linewidth]{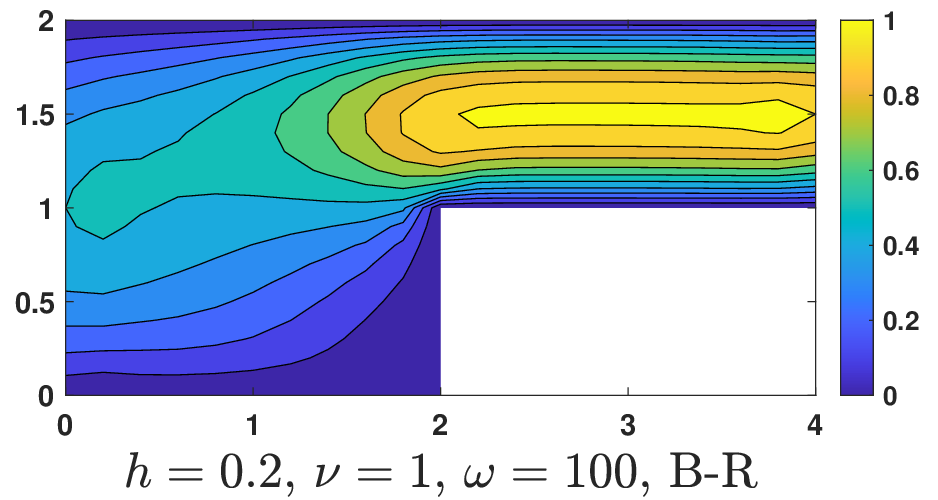}\vspace{4pt}
    \includegraphics[width=1\linewidth]{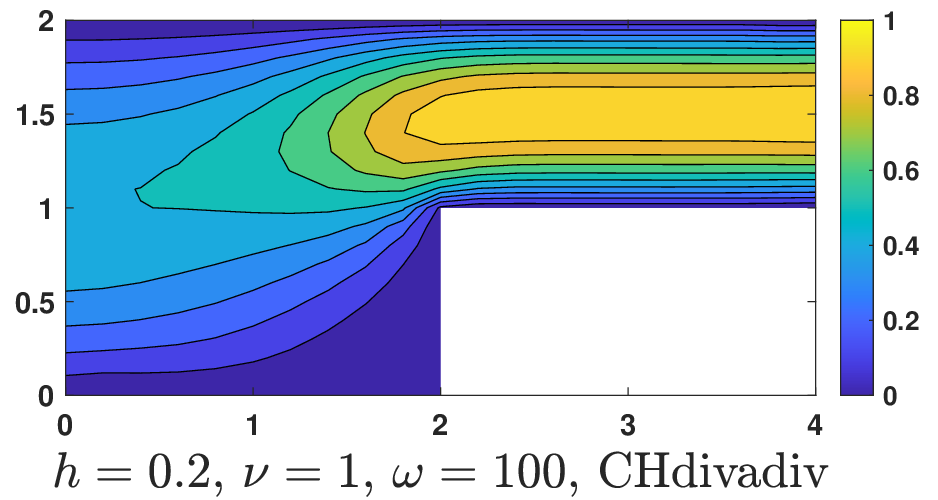}
    \end{minipage}}
 \subfigure{
    \begin{minipage}[b]{0.23\linewidth}
    \includegraphics[width=1\linewidth]{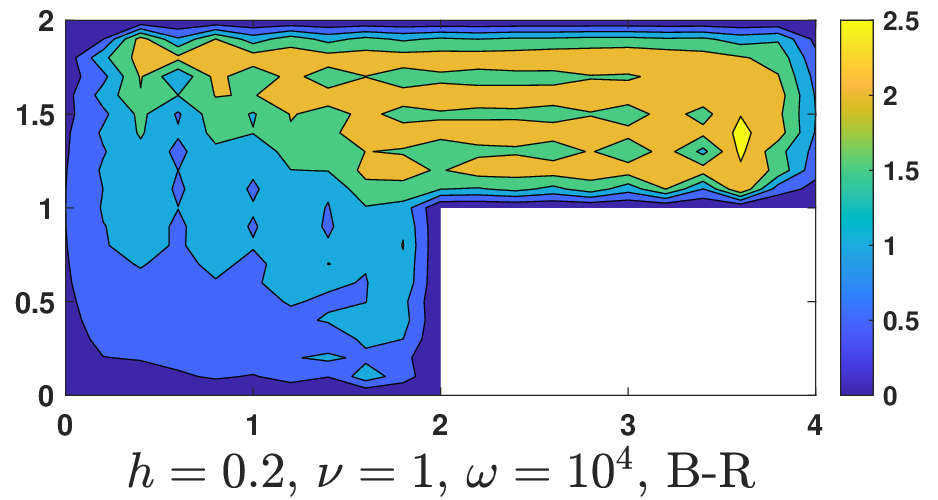}\vspace{4pt}
    \includegraphics[width=1\linewidth]{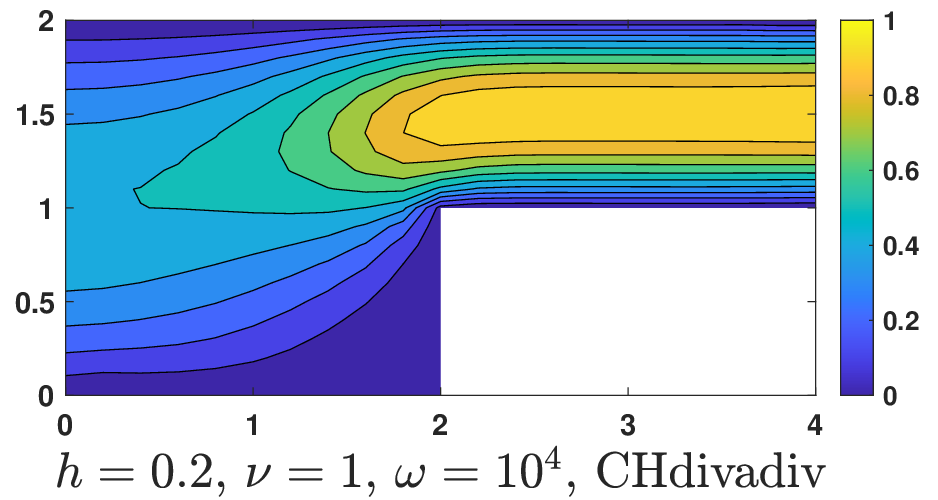}
    \end{minipage}}
 \subfigure{
    \begin{minipage}[b]{0.23\linewidth}
    \includegraphics[width=1\linewidth]{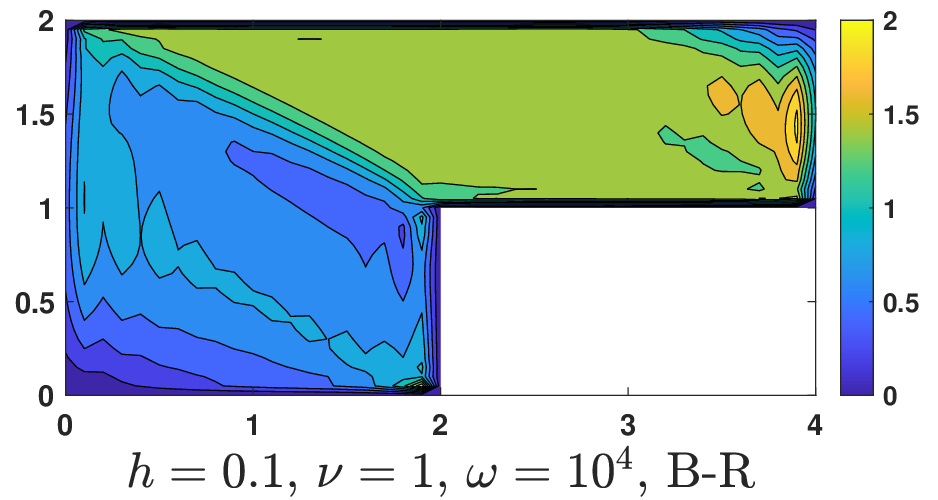}\vspace{4pt}
    \includegraphics[width=1\linewidth]{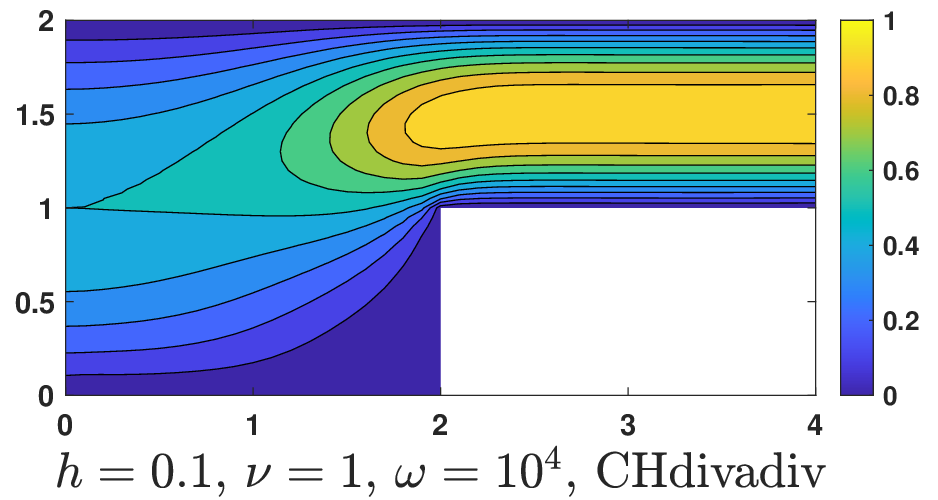}
    \end{minipage}}
    \caption{\cref{EX:53}. Absolute value of the velocity (speed) obtained by the Bernardi and Raugel
   element \cite{bernardi_analysis_1985} (top) and the method \cref{formula2} with $a^{R}=a^{\operatorname{div}}$ (bottom).}
\label{fig54}
\end{figure}

One could see, our method is robust with respect to the magnitude $\omega$ of the Coriolis force, while undesired instabilities occur for the Bernardi and Raugel element when a strong Coriolis force exists ($\omega=100,10000$).

\subsection{The robustness of the pressure-robust ${P_{1}^{c}-P0}$ scheme}
\label{EX:54}
In this example we want to show that the construction method of the ${\bm P_{1}^{c}-P0}$ discretization on the Stokes equations could be extended to general incompressible fluid flows. We will use the Brinkman equation and the Coriolis force problem (like \cref{EX:53} but the exact velocity is known) as two examples. The (scaled) Brinkman equation is described as
\begin{equation}\label{Brinkman}
-\nu\Delta\boldsymbol{u}+\boldsymbol{u}+\nabla p=\boldsymbol{f}, \quad \nabla\cdot\boldsymbol{u}=0\quad \text{in} ~\Omega,
\end{equation}
and the Coriolis force problem is described as \cref{Coriolis}. In two problems the exact velocity are the same as \cref{EX:51} and we apply the homogenous velocity boundary conditions. In \cref{Brinkman} we choose $\nu=10^{-6}$ and the pressure is also the same as \cref{EX:51}. In the Coriolis problem we choose $\nu=1$ and $\omega=10^{4}$ with $\boldsymbol{f}\equiv0$ ({resulting in a very large pressure gradient that balances the Coriolis force}). The corresponding discretizations respectively read: (in $a_{h}(\cdot,\cdot)$ we choose $a^{R}=a^{\operatorname{div}}$ with $\alpha_{T}\equiv 1.5$)
{\begin{subequations}\nonumber
\begin{align}
{\rm Find}~(\boldsymbol{u}_{h},p_{h})\in V_{h}\times W_{h}~ {\rm such ~that}\nonumber~~~~~~~~&\\
\nu a_{h}(\boldsymbol{u}_{h}, \boldsymbol{v}_{h})+a^{B}(\boldsymbol{u}_{h}, \boldsymbol{v}_{h})-b(\boldsymbol{v}_{h}, p_{h}) &+b(\boldsymbol{u}_{h}, q_{h})=(\boldsymbol{f}, \boldsymbol{v}_{h})
\quad\forall~\boldsymbol{v}_{h}\in V_{h}, ~q_{h}\in W_{h},
\end{align}
\end{subequations}}
and
{\begin{subequations}\nonumber
\begin{align}
{\rm Find}~(\boldsymbol{u}_{h},p_{h})\in V_{h}\times W_{h}~ {\rm such ~that}\nonumber~~~~~~~~&\\
\nu a_{h}(\boldsymbol{u}_{h}, \boldsymbol{v}_{h})+a^{C}(\boldsymbol{u}_{h}, \boldsymbol{v}_{h})-b(\boldsymbol{v}_{h}, p_{h}) &+b(\boldsymbol{u}_{h}, q_{h})=(\boldsymbol{f}, \boldsymbol{v}_{h})
\quad\forall~\boldsymbol{v}_{h}\in V_{h}, ~q_{h}\in W_{h},
\end{align}
\end{subequations}}
where
{\begin{displaymath}
a^{B}(\boldsymbol{u}_{h}, \boldsymbol{v}_{h})=(\boldsymbol{u}_{h}^{1}, \boldsymbol{v}_{h}^{1})+(\boldsymbol{u}_{h}^{R}, \boldsymbol{v}_{h}^{1})+(\boldsymbol{u}_{h}^{1}, \boldsymbol{v}_{h}^{R})+(d+1)\sum_{T\in\mathcal{T}_{h}}\sum_{e\in\partial T\cap\mathcal{E}^{0}}u_{e}v_{e}(\boldsymbol{\Phi}_{e},\boldsymbol{\Phi}_{e})_{T},
\end{displaymath}}
and
$$a^{C}(\boldsymbol{u}_{h}, \boldsymbol{v}_{h})=(2\boldsymbol{\omega}\times\boldsymbol{u}_{h}^{1}, \boldsymbol{v}_{h}^{1})+(2\boldsymbol{\omega}\times\boldsymbol{u}_{h}^{R}, \boldsymbol{v}_{h}^{1})+(2\boldsymbol{\omega}\times\boldsymbol{u}_{h}^{1}, \boldsymbol{v}_{h}^{R}).$$
The final term in $a^{B}$ is to preserve the stability since the second term and the third term in $a^{B}$ are two indefinite terms. {A similar term} does not exist in $a^{C}$ since one could verify that $a^{C}$ itself has a skew-symmetric structure, that is, $a^{C}(\boldsymbol{v}_{h},\boldsymbol{v}_{h})=0$ for all $\boldsymbol{v}_{h}\in V_{h}$. {This time} the $RT0$-$RT0$ regions of the coefficient matrices are both diagonal {for the Brinkman and Coriolis force equations}. Thus, we can easily obtain the corresponding ${\bm P_{1}^{c}-P0}$ schemes {via static condensation}. Unlike the {scheme for the Stokes equations}, in these two cases, the $\boldsymbol{u}_{h}^{R}$ parts are expressed in terms of $\boldsymbol{u}_{h}^{1}$ and $p_{h}$ together, not only $p_{h}$. Some convergence results are shown in \cref{tab:2} and \cref{tab:3}.
\begin{table}
\centering
\caption{\cref{EX:54}. Errors for the Brinkman equation with $\nu=10^{-6}$.}
\label{tab:2}
\begin{tabular}{cllllll}
\hline h& $||\boldsymbol{u}-\boldsymbol{u}_{h}||$ & order & $||\nabla(\boldsymbol{u}-\boldsymbol{u}_{h}^{1})||$ & order & $||p-p_{h}||$ & order \\
\hline
1E-1& 1.23E-2   &      & 1.15    &  & 2.51E-2 &  \\
5E-2& 2.93E-3   & 2.07 & 5.73E-1 & 1.01 & 1.26E-2 & 0.99 \\
2.5E-2& 7.10E-4 & 2.04 & 2.85E-1 & 1.00 & 6.32E-3 & 0.99 \\
1.25E-2&1.74E-4 & 2.02  & 1.42E-1 & 1.00    & 3.16E-3 & 0.99 \\
6.25E-3& 4.32E-5& 2.01  & 7.11E-2 & 1.00    & 1.58E-3 & 0.99 \\
\hline
\end{tabular}
\end{table}
\begin{table}
\centering
\caption{\cref{EX:54}. Errors for the Stokes equation with strong Coriolis force ($\omega=10^{4}$).}
\label{tab:3}
\begin{tabular}{cllll}
\hline h& $||\boldsymbol{u}-\boldsymbol{u}_{h}||$ & order & $||\nabla(\boldsymbol{u}-\boldsymbol{u}_{h}^{1})||$ & order  \\
\hline
1E-1& 4.27E-2   &      & 1.18    &   \\
5E-2& 8.41E-3   & 2.34 & 5.71E-1 & 1.04  \\
2.5E-2& 1.37E-3 & 2.61 & 2.84E-1 & 1.00  \\
1.25E-2&3.19E-4 & 2.10  & 1.42E-1 & 1.00     \\
6.25E-3& 8.11E-5& 1.97  & 7.11E-2 & 1.00     \\
\hline
\end{tabular}
\end{table}

\section{Conclusions}
\label{sec:6}
We have developed a class of divergence-free mixed methods for the Stokes problem based on ${\bm P_{1}^{c}}\oplus RT0-P0$ pair. By putting this pair into a modified conforming-like formulation, our methods combine some advantages of conforming methods and $H(\operatorname{div})$-conforming methods. In addition, a stabilized and pressure-robust ${\bm P_{1}^{c}}-P0$ discretization is proposed in this paper which is computationally cheap. Future work includes:
\begin{itemize}
  \item Extending our methods to the mixed velocity and stress boundary value problems which are also an important class of problems in real-world applications. {We note that, in this extension, a minor modification to our methods is required to preserve the convergence speed.}
  \item {Extending our methods to the Navier-Stokes equations (NSEs). The key issue is the discretization of the nonlinear term in NSE. In this regard, we will propose a modified convective formulation which conserves the energy, linear momentum and angular momentum in some cases. This is based on the construction of our velocity space.}
  \item We are planning to propose a general framework on constructing stable and pressure-robust ${\bm P_{1}^{c}}-P0$ discretizations for incompressible fluid flows, including the Brinkman equations, the Oseen equations, the Navier-Stokes equations etc. Some related numerical results have been shown in this paper (see \cref{EX:54}).
\end{itemize}

\section*{Funding}

This work was supported by the National Natural Science Foundation of China grant {12131014}.

\bibliographystyle{amsplain}
\bibliography{references}
\end{document}